\documentclass{amsart}
\usepackage{graphicx} 
\usepackage{amsmath}
\usepackage{amsfonts, amssymb}
\usepackage{amsthm}
\usepackage{mathabx}
\usepackage[dvipsnames]{xcolor}
\usepackage{xcolor}
\usepackage{tikz-cd}
\usetikzlibrary{decorations.markings,arrows,automata,arrows,backgrounds}
\usepackage{tikz}
\usetikzlibrary{arrows.meta, decorations.markings, calc}
\usetikzlibrary{calc}
\usepackage{mathrsfs}
\usepackage{algorithm}
\usepackage{algpseudocode}

\theoremstyle{definition}
\newtheorem{theorem}{Theorem}[section]

\newtheorem{proposition}[theorem]{Proposition}
\newtheorem{definition}[theorem]{Definition}
\newtheorem{lemma}[theorem]{Lemma}
\newtheorem{remark}[theorem]{Remark}

\newtheorem{example}[theorem]{Example}

\definecolor{hexblue}{HTML}{0021A5}
\definecolor{fgreen}{RGB}{34,139,34}

\newcommand{\zr}{{\mathbb R}}
\newcommand{\zz}{{\mathbb Z}}

\newcommand{\zs}{{\mathbb S}}
\newcommand{\sm}{{\mathcal{SM}}}

\newcommand{\smpure}{{\mathcal{SM}}_{\textrm{pure}}}

\newcommand\sco{{\searrow \hspace{-6 pt}\searrow\hspace{3 pt}}}

\title{Complexes of strong discrete Morse matchings}
\author{Kevin P.~Knudson}
\author{Abigail Owens--White}
\address{Department of Mathematics, University of Florida, Gainesville, FL 32611}

\date{\today}

\email{kknudson@ufl.edu}
\email{aowens2@ufl.edu}


\begin{document}

\begin{abstract}
Using the strong discrete Morse theory developed by Fern\'andez \cite{Fernandez2026}, we define the complex $\sm(K)$ of strong discrete Morse matchings on a simplicial complex $K$, as well as the pure subcomplex $\smpure(K)$ generated by the facets in $\sm(K)$ of maximal dimension. For most complexes $K$ these objects are proper subcomplexes of the complexes ${\mathfrak M}(K)$ and ${\mathfrak M}_{\textrm{pure}}(K)$ defined by Chari--Joswig \cite{ChariJoswig2005} using all Morse matchings on $K$. The homotopy types of the latter are not well-understood in general, but they are known when $K$ is the path $P_n$ with $n$ edges, the cycle $C_n$ with $n$ edges, the star $S_n$ with $n$ leaves, the $n$-simplex $\Delta^n$ ($n\le 3$), and the boundary $\partial\Delta^n$ ($n\le 3$). in this paper we compute the homotopy types of $\smpure(C_n)$ and $\sm(K)$ for $K=P_n,S_n,\Delta^n,\partial\Delta^n$ for all $n$. We also compute the homology of $\sm(C_n)$ for $n\le 21$.
\end{abstract}

\maketitle

\section{Introduction}  

Discrete Morse theory was developed by Forman \cite{Forman1998} as a combinatorial analogue of classical smooth Morse theory on manifolds. A discrete Morse function on a simplicial complex $K$ is an assignment $f$ of a real number to each simplex of $K$ that generically increases with dimension in the sense that any $p$-simplex $\sigma$ in $K$ has at most one $(p+1)$-simplex $\tau$ containing it with $f(\tau)\le f(\sigma)$. These pairs determine a partial matching on the Hasse diagram of $K$ called the discrete gradient vector field of $f$ or a Morse matching; reversing those directed edges in the matching yields an acyclic directed graph. Conversely, any acyclic partial matching on the Hasse diagram is the gradient of a discrete Morse function on $K$ (infinitely many, in fact). Moreover, collapsing the pairs in decreasing order as defined by $f$ transforms $K$ into a simpler simplicial complex with the same homotopy type. Discrete Morse theory is therefore a powerful tool for topological simplification.

In \cite{ChariJoswig2005}, Chari and Joswig define and analyze $\mathfrak{M}(K)$, the \textit{complex of discrete Morse matchings}, along with the subcomplex ${\mathfrak M}_\textrm{pure}(K)$ generated by the facets of maximal dimension. The former is the simplicial complex whose simplices are the Morse matchings on the complex $K$; that is, the vertices are the edges of the Hasse diagram of $K$ and the simplices are those sets of edges that give Morse matchings upon their reversal. For $K$ a graph, $\mathfrak{M}(K)$ is exactly the complex of rooted forests on $K$. Kozlov \cite{Kozlov1999} studied these; the case of a double-directed cycle gives the homotopy type of $\mathfrak{M}(C_n)$ and the homotopy type depends on the residue of $n \bmod 3$. Chari and Joswig independently found that $\mathfrak{M}_{pure}(C_n)$ is homotopic to $\mathbb{S}^2\vee\mathbb{S}^{n-2}\vee\mathbb{S}^{n-2}$. They also compute the homotopy type of $\mathfrak{M}(\Delta^1)$ and $\mathfrak{M}(\Delta^2)$, and compute the f-vector and homology of $\mathfrak{M}(\Delta^3)$. Recent work of Scoville \cite{Scoville2026} determines the homotopy type of ${\mathfrak M}(\Delta^3)$, but beyond these results, explicit calculations of the homotopy type of $\mathfrak{M}(K)$ remain elusive.

In \cite{Fernandez2026}, Fern\'andez introduces \textit{strong discrete Morse theory} as another tool to simplify simplicial complexes while preserving homotopy type. This variation of discrete Morse theory utilizes the strong collapses for simplicial complexes studied in \cite{BarmakMinian2009} by Barmak and Minian. Building on Whitehead's collapses and simple homotopy theory, Barmak and Minian's \textit{strong homotopy theory} defines strong collapses to eliminate \textit{dominated vertices} and their open stars. Dominated vertices are defined as those vertices in a simplicial complex $K$ whose link in $K$ is a simplicial cone. Strong collapses preserve the homotopy type of $K$ and produce a minimal irreducible subcomplex, called the \textit{strong core} of $K$. An arbitrary simplicial complex may not have any dominated vertices, so Fern\'andez uses filtrations of $K$ to define the concept of an \textit{internal strong collapse} in $K$. This process finds and collapses dominated vertices in subcomplexes of $K$ in a systematic way and the main result of \cite{Fernandez2026} is that any sequence of these internal strong collapses yields a {\em regular} CW-complex that is homotopy equivalent to $K$.

In this paper, we present analogous results to those of Chari and Joswig, but in the realm of Fern\'andez's strong discrete Morse theory. We first define the \textit{complex of strong discrete Morse matchings} $\sm(K)$ as the simplicial complex generated by all strong discrete Morse matchings on $K$. This complex has a pure counterpart, denoted $\smpure(K)$, generated by the facets of $\sm(K)$ of maximal dimension. The main examples that we consider are $K=P_n,C_n,\Delta^n,$ and $\partial\Delta^n$. In Theorem \ref{pathhtyptype} we compute the homotopy type of $\sm(P_n)$:
$$\sm(P_n) \simeq \begin{cases}
    \zs^{2k-1} & n=4k \\
    \zs^{2k} & n=4k+1 \\
    \ast & n=4k+2,4k+3.
\end{cases}$$
In Theorem \ref{thm:cnpure}, compute the homotopy type of $\smpure(C_n)$, with a pattern emerging for $n\geq6$:
\begin{equation*}
\smpure(C_n)\simeq
    \begin{cases}
        \mathbb{S}^{2}\vee\mathbb{S}^{2k-2}\vee\mathbb{S}^{2k-2}\vee\mathbb{S}^{2k-2}\vee\mathbb{S}^{2k-2} & \text{if } n=3k\\
        \mathbb{S}^{2}\vee\mathbb{S}^{2k-1}\vee\mathbb{S}^{2k-1} & \text{if } n=3k+1, 3k+2\\
    \end{cases}
\end{equation*}
This computation is surprisingly complicated, much more involved than Chari--Joswig's computation of the homotopy type of ${\mathfrak M}_\textrm{pure}(C_n)$. The homotopy types of $\sm(C_n)$ are much more difficult to compute, and follow no discernible pattern. We instead give the homology types of this complex up to $n=21$.

The final examples we consider are the $n$-simplex and its boundary. We find that $\sm(\Delta^n)$ has the homotopy type of a wedge of $n^n$ copies of $\zs^{n-1}$, and $\sm(\partial\Delta^n)$ has the homotopy type of a wedge of $f(n)$ copies of $\zs^{n-2}$, where
\begin{equation*}
    f(n) = \frac{n^2 + 3n}{2} + \sum_{k=1}^{n-2} \binom{n+1}{k}(n-k)^{n-k}.
\end{equation*}
These computations are nontrivial, but given that the homotopy types of ${\mathfrak M}(\Delta^n)$ and ${\mathfrak M}(\partial\Delta^n)$ remain a mystery for $n\ge 4$ it is curious that we are able to understand the strong collapses relatively easily.
The final section of the paper considers for $\sm(K)$ various related results analogous to those known for ${\mathfrak M}(K)$. We also compute the homotopy types of some incidence complexes related to some of our $\sm(K)$.

\section{Discrete Morse Theory}

We provide a quick review of discrete Morse theory before introducing the strong version. Recall that for $K$ a finite simplicial complex, a \textit{discrete Morse function} $f:K\to\mathbb{R}$ satisfies the following two conditions, where $\alpha^{(k)}$ is a simplex of dimension $k$:
\begin{enumerate}
\item $|\{\beta^{(k+1)}>\alpha^{(k)}:f(\beta)\leq f(\alpha)\}|\leq 1$
\item $|\{\gamma^{(k-1)}<\alpha^{(k)}:f(\alpha)\leq f(\gamma)\}|\leq 1$
\end{enumerate}
To such a function we have an associated \textit{gradient vector field} $V_f$. This consists of the pairs $\{\alpha^{(k)}<\beta^{(k+1)}\}$ with $f(\alpha)\ge f(\beta)$. Unpaired simplices are called \textit{critical}. As mentioned in the Introduction, one may reinterpret this collection as a directed graph by taking the Hasse diagram of $K$ with edges pointing down in dimension and reversing those arrows corresponding to pairs in $V_f$. The resulting directed graph is acyclic, and conversely, given such an acyclic partial matching there is a discrete Morse function whose gradient is the matching. This imposes an equivalence relation on the set of discrete Morse functions on $K$, so one typically works with the gradient vector fields instead. These gradients are also called {\em Morse matchings} on $K$.

There are two main theorems of discrete Morse theory presented in \cite{Forman1998}. Denote by $K(a)$ the subcomplex of $K$ generated by all simplices $\tau$ with $f(\tau)\leq a$. The first is that for any interval $[a,b]$ containing no critical simplices, $K(a)$ is a deformation retract of $K(b)$, and in fact there is a simplicial collapse $K(b)\searrow K(a)$. The second states that for an interval $[a,b]$ containing a single critical simplex $\sigma^{(p)}$, $K(b)$ is homotopy equivalent to $$K(a)\bigcup e^{(p)}$$ where $e^{(p)}$ is attached to $K(a)$ along its boundary. As a corollary, any simplicial complex with a discrete Morse function is then homotopy equivalent to a CW-complex with exactly one cell of dimension $p$ for every critical cell of dimension $p$.

\section{Strong discrete Morse theory}\label{sec:strong}

Strong discrete Morse theory, developed in \cite{Fernandez2026}, adapts Forman's discrete Morse theory by using the combinatorial tools of \cite{BarmakMinian2009} to define \textit{internal strong collapses} on a simplicial complex $K$. The result of these collapses is a regular CW-complex called a \textit{strong core} of $K$. This strong core is not unique, but it will always preserve the homotopy type of $K$. Here we present the main results from \cite{Fernandez2026}.

Let $K$ be a simplicial complex with vertex $v$. The \textit{link} of $v$ in $K$, denoted $\text{lk}(v,K)$ or simply $\text{lk}(v)$ if the complex $K$ is clear, is the subcomplex of $K\setminus v$ consisting of simplices $\sigma$ such that $\sigma\cup\{v\}$ is a simplex of $K$. A \textit{simplicial cone} with apex $a\not\in V(K)$ is the complex $a*K$ consisting of all simplices of $K$, the vertex $a$, and all simplices of the form $\sigma\cup\{a\}$ for $\sigma\in K$. The \textit{star} of a simplex $\sigma$,  $\text{st}_K(\sigma)$, consists of all simplices of $K$ containing $\sigma$; its closure will be denoted $\text{St}_K(\sigma)$. Again, we omit the subscript $K$ if the context is clear. Note that stars (both open and closed) are contractible. We denote by $K\setminus v$ the subcomplex obtained from $K$ by removing $\text{st}_K(v)$.

\begin{definition}[\cite{BarmakMinian2009}]
Let $K$ be a simplicial complex. A vertex $v$ of $K$ is \textit{dominated} by a vertex $a$ if $\text{lk}(v,K)=a*K_0$ where $K_0$ is a subcomplex of $K$. In that case, there is an \textit{elementary strong collapse} from $K$ to $K\setminus v$, denoted $K~\sco K\setminus v$. A \textit{strong collapse} $K~\sco L$ from a simplicial complex $K$ to a subcomplex $L$ is a sequence of elementary strong collapses starting at $K$ and ending at $L$. In that case, we say $L$ is a \textit{strong core} of $K$. If L has no dominated vertices, we say that $L$ is a \textit{minimal strong core} of $K$. 
\end{definition}

Some simplicial complexes do not have dominated vertices. A simple example is the $n$-cycle $C_n$, the $n$-vertex triangulation of a circle. The link of any vertex in $C_n$ consists of two points and is therefore not a cone. Strong discrete Morse theory remedies this via vertex filtrations. A \textit{vertex filtration} is an injective map $g:V(K)\to\mathbb{R}$ on the set of vertices of $K$ that assigns a real value to each $v\in V(K)$. We can then filter $K$ by subcomplexes generated by vertices with function values bounded by some threshold. This idea is formalized in the following definition.

\begin{definition} (\cite{Fernandez2026}, Definition 3.5)
Let $g:V(K)\to\mathbb{R}$ be a vertex filtration. If $v$ is a vertex in $K$, let $K_{g(v)}$ denote the subcomplex of $K$ generated by all vertices with function value at most $g(v)$. The \textit{descending open star} of $v$, denoted $\textrm{st}^\downarrow(v,K)$ is the open star of $v$ in the subcomplex $K_{g(v)}$. Similarly, the \textit{descending link} of $v$, denoted $\textrm{lk}^\downarrow(v,K)$, is the link of $v$ in $K_{g(v)}$. We say that $v$ is \textit{descending dominated} if it is dominated in $K_{g(v)}$ by a vertex $a$ with $g(a)<g(v)$. Otherwise, $v$ is called a \textit{strong critical vertex}.
\end{definition}

 Given a vertex filtration $g:V(K)\to\zr$, we build an associated \textit{strong matching $\mathcal{M}$} as follows. Initialize ${\mathcal M}=\emptyset$. Order the vertices by increasing function value and proceed from the minimal vertex. Its descending link is empty and so the vertex is critical. Examine each vertex in turn; if a vertex $u$ is not descending dominated then it remains critical, as does the entire subcomplex $K_{g(u)}$. Suppose $v$ is the first vertex that is descending dominated in $K_{g(v)}$, say by vertex $a$. Then for $\sigma\in\text{St}^\downarrow(v)$, we add all pairs of the form $\{\sigma,\sigma\cup\{a\}\}$ to ${\mathcal M}$. In particular, $v$ gets paired with the edge $[v,a]$. Proceed through the remaining vertices, updating the matching each time a descending dominated vertex is reached.  Note that strong critical vertices remain unpaired as expected, as do their descending open stars. Moreover, this process is not unique. Indeed, it may happen that a vertex $v$ is descending dominated by more than one vertex in $K_{g(v)}$. In that case, we simply make a choice of the dominating vertex and update the matching. Thus, for a given filtration function $g$, there may be many associated matchings. Figure \ref{fig:triangles} shows an example of four different strong matchings arising from the same vertex filtration. Pseudocode for this process is shown in Algorithm \ref{alg:internal_core_computation}. 

\begin{theorem} (\cite{Fernandez2026}, Theorem 3.9)
    Given a vertex filtration $g:V(K)\to\zr$, an associated matching ${\mathcal M}$ is acyclic.
\end{theorem}
 
\begin{figure}
\centering
\begin{tikzpicture}[
    >={Latex[length=2.5mm]},
    line/.style={thick, black},
    vtx/.style={circle,fill=black, inner sep=1.6pt},
    hexedge/.style={thick, black},
    arrowedge/.style={thick, ->, color=hexblue},scale=0.9
]

\tikzstyle{black state}=[
    circle,
    draw = black,
    thick,
    fill = black,
    minimum size = 1mm
]
\tikzstyle{red state}=[
    circle,
    draw = red,
    thick,
    fill = red,
    minimum size = 1mm
]
 
\newcommand{\baseTriangle}{
    \coordinate (0)  at (1.5,2.5);
    \coordinate (P1) at (0,0);
    \coordinate (P2) at (1.5,0);
    \coordinate (P3) at (3,0);
    \coordinate (12) at (0.75,0);
    \coordinate (23) at (2.25,0);
    \coordinate (01) at (0.75,1.25);
    \coordinate (02) at (1.5,1.25);
    \coordinate (03) at (2.25,1.25);

    \fill[black!15] (P1) -- (0) -- (P2) -- cycle;
    \fill[black!15] (P2) -- (0) -- (P3) -- cycle;

    \draw[line] (P1) -- (0);
    \draw[line] (P2) -- (0);
    \draw[line] (P3) -- (0);
    \draw[line] (P1) -- (P2);
    \draw[line] (P2) -- (P3);

    \node[vtx, red state]   at (0)  {};
    \node[vtx, black state] at (P1) {};
    \node[vtx, black state] at (P2) {};
    \node[vtx, black state] at (P3) {};

    \node at ($(0)+(0,0.4)$)    {0};
    \node at ($(P1)+(-0.35,0)$) {1};
    \node at ($(P2)+(0,-0.4)$)  {2};
    \node at ($(P3)+(0.35,0)$)  {3};
}
 
\begin{scope}[shift={(0,0)}]
\baseTriangle 
\draw[arrowedge] ($(P1)!0!(0)$) to ($(P1)!0.50!(0)$);
\draw[arrowedge] ($(P2)!0!(0)$) to ($(P2)!0.50!(0)$);
\draw[arrowedge] ($(P3)!0!(0)$) to ($(P3)!0.50!(0)$);
\draw[arrowedge] ($(12)!0!(0)$) to ($(12)!0.30!(0)$);
\draw[arrowedge] ($(23)!0!(0)$) to ($(23)!0.30!(0)$);
\end{scope}

\begin{scope}[shift={(5,0)}]
\baseTriangle;
\draw[arrowedge] ($(P1)!0!(0)$) to ($(P1)!0.50!(0)$);
\draw[arrowedge] ($(P2)!0!(0)$) to ($(P2)!0.50!(0)$);
\draw[arrowedge] ($(P3)!0!(P2)$) to ($(P3)!0.50!(P2)$);
\draw[arrowedge] ($(12)!0!(0)$) to ($(12)!0.30!(0)$);
\draw[arrowedge] ($(03)!0!(P2)$) to ($(03)!0.50!(P2)$);
\end{scope}
 
\begin{scope}[shift={(0,-5)}]
\baseTriangle;
\draw[arrowedge] ($(P1)!0!(0)$) to ($(P1)!0.50!(0)$);
\draw[arrowedge] ($(P2)!0!(P1)$) to ($(P2)!0.50!(P1)$);
\draw[arrowedge] ($(P3)!0!(0)$) to ($(P3)!0.50!(0)$);
\draw[arrowedge] ($(02)!0!(P1)$) to ($(02)!0.40!(P1)$);
\draw[arrowedge] ($(23)!0!(0)$) to ($(23)!0.30!(0)$);
\end{scope}
 
\begin{scope}[shift={(5,-5)}]
\baseTriangle;
\draw[arrowedge] ($(P1)!0!(0)$) to ($(P1)!0.50!(0)$);
\draw[arrowedge] ($(P2)!0!(P1)$) to ($(P2)!0.50!(P1)$);
\draw[arrowedge] ($(P3)!0!(P2)$) to ($(P3)!0.50!(P2)$);
\draw[arrowedge] ($(02)!0!(P1)$) to ($(02)!0.40!(P1)$);
\draw[arrowedge] ($(03)!0!(P2)$) to ($(03)!0.50!(P2)$);
\end{scope}
 
\end{tikzpicture}
\caption{The strong gradient from a vertex ordering is not unique. In this complex, vertex $2$ is descending dominated by both $0$ and $1$. Thus, it can be paired with the edge $12$ or $02$. That choice determines which edge pairs with the triangle $012$. A similar situation happens with vertex $3$. All of these strong gradients still provide a matching that captures the correct homotopy type of the simplicial complex.}
\label{fig:triangles}
\end{figure}
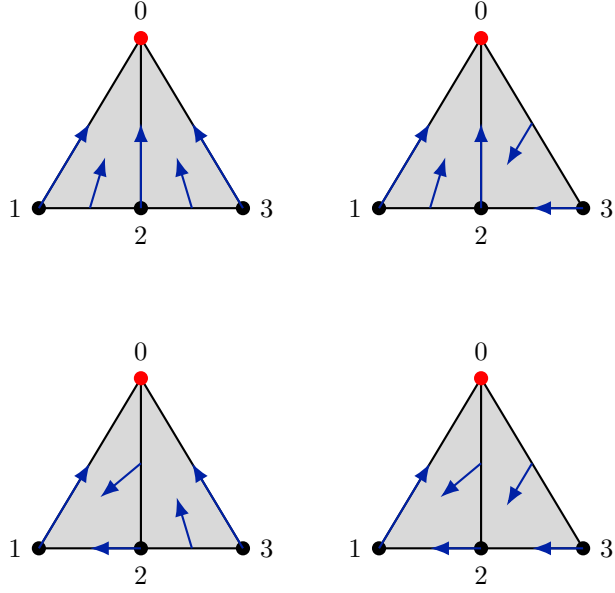

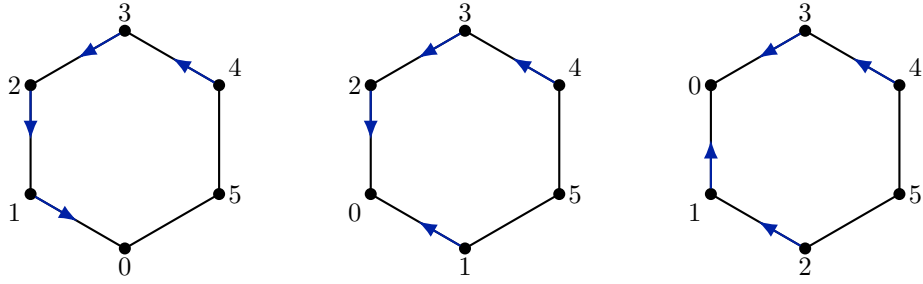
\begin{figure}
\centering
\begin{tikzpicture}[
    >={Latex[length=2.5mm]},
    vtx/.style={circle, fill=black, inner sep=1.6pt},
    hexedge/.style={thick, black},
    arrowedge/.style={thick, ->, color=hexblue},scale=0.9
]
 
\coordinate (T1)  at (0,1.6);
\coordinate (TR1) at (1.3856,0.8);
\coordinate (R1)  at (1.3856,-0.8);
\coordinate (B1)  at (0,-1.6);
\coordinate (BL1) at (-1.3856,-0.8);
\coordinate (TL1) at (-1.3856,0.8);
 
\coordinate (T2)  at (5,1.6);
\coordinate (TR2) at (6.3856,0.8);
\coordinate (R2)  at (6.3856,-0.8);
\coordinate (B2)  at (5,-1.6);
\coordinate (BL2) at (3.6144,-0.8);
\coordinate (TL2) at (3.6144,0.8);
 
\coordinate (T3)  at (10,1.6);
\coordinate (TR3) at (11.3856,0.8);
\coordinate (R3)  at (11.3856,-0.8);
\coordinate (B3)  at (10,-1.6);
\coordinate (BL3) at (8.6144,-0.8);
\coordinate (TL3) at (8.6144,0.8);
 
\draw[hexedge] (T1)--(TR1)--(R1)--(B1)--(BL1)--(TL1)--cycle;
\draw[hexedge] (T2)--(TR2)--(R2)--(B2)--(BL2)--(TL2)--cycle;
\draw[hexedge] (T3)--(TR3)--(R3)--(B3)--(BL3)--(TL3)--cycle;
 
\foreach \c in {T1,TR1,R1,B1,BL1,TL1,T2,TR2,R2,B2,BL2,TL2,T3,TR3,R3,B3,BL3,TL3}
    \node[vtx] at (\c) {};
 
\node[above]       at (T1)  {$3$};
\node[above right] at (TR1) {$4$};
\node[right]        at (R1)  {$5$};
\node[below]       at (B1)  {$0$};
\node[below left]  at (BL1) {$1$};
\node[left]         at (TL1) {$2$};
 
\node[above]       at (T2)  {$3$};
\node[above right] at (TR2) {$4$};
\node[right]        at (R2)  {$5$};
\node[below]       at (B2)  {$1$};
\node[below left]  at (BL2) {$0$};
\node[left]         at (TL2) {$2$};
 
\node[above]       at (T3)  {$3$};
\node[above right] at (TR3) {$4$};
\node[right]        at (R3)  {$5$};
\node[below]       at (B3)  {$2$};
\node[below left]  at (BL3) {$1$};
\node[left]         at (TL3) {$0$};
 

\draw[arrowedge] ($(TR1)!0.05!(T1)$) to ($(TR1)!0.50!(T1)$);
\draw[arrowedge] ($(T1)!0.05!(TL1)$)  to ($(T1)!0.50!(TL1)$);
\draw[arrowedge] ($(TL1)!0.05!(BL1)$) to ($(TL1)!0.50!(BL1)$);
\draw[arrowedge] ($(BL1)!0.05!(B1)$)  to ($(BL1)!0.50!(B1)$);
 
\draw[arrowedge] ($(TR2)!0.05!(T2)$)  to ($(TR2)!0.50!(T2)$);
\draw[arrowedge] ($(T2)!0.05!(TL2)$)  to ($(T2)!0.50!(TL2)$);
\draw[arrowedge] ($(TL2)!0.05!(BL2)$) to ($(TL2)!0.50!(BL2)$);
\draw[arrowedge] ($(BL2)!0.95!(B2)$) to ($(BL2)!0.50!(B2)$);
 
\draw[arrowedge] ($(TR3)!0.05!(T3)$)  to ($(TR3)!0.50!(T3)$);
\draw[arrowedge] ($(T3)!0.05!(TL3)$)  to ($(T3)!0.50!(TL3)$);
\draw[arrowedge] ($(TL3)!0.95!(BL3)$) to ($(TL3)!0.50!(BL3)$);
\draw[arrowedge] ($(BL3)!0.95!(B3)$) to ($(BL3)!0.50!(B3)$);
 
\end{tikzpicture}
\caption{$C_6$ with three different vertex filtrations and the resulting strong Morse matchings. Up to the symmetries of $C_6$, these are all of the maximal strong Morse matchings on $C_6$.}
\label{fig:C_6}
\end{figure}

\begin{example}
Figure \ref{fig:C_6} gives a simple example of three different vertex filtrations on $K=C_6$. We build the strong Morse matching related to the vertex filtration by looking at each subcomplex $K_{g(v)}$ of $K$. For the leftmost hexagon, the first subcomplex is $K_0$, where we have a single strong critical vertex $0$. The complex $K_1$ consists of vertices $1,0$, and the edge $01$ and then lk$^\downarrow(1,K)=\{0\}\simeq *$, a simplicial cone. Thus $1$ is descending dominated by $0$, and we pair $1$ with the edge $01$. We update the matching as we examine the subcomplexes $K_2,K_3,$ and $K_4$. The final subcomplex is always the entire complex; here, it is $K_5=K$. Note that the lower link of vertex $5$ is the disjoint union of vertices $0$ and $4$; this is simply $\zs^0$ and this is not a simplicial cone. So vertex $5$ is not downward dominated and remains strong critical. The other two copies of $C_6$ have different vertex filtrations yielding distinct strong Morse matchings.
\end{example}

It is clear from Figures \ref{fig:triangles} and \ref{fig:C_6} that we can have many different strong Morse matchings on a given simplicial complex. For any of them, the sequence of internal collapses yields a regular CW-complex.

\begin{theorem}[\cite{Fernandez2026}, Theorem 3.7, Strong Discrete Morse Theory]
Let $K$ be a finite simplicial complex, and let $g:V(K)\to\mathbb{R}$ be a real-valued function. Then $K$ is homotopy equivalent to a regular CW-complex core$_g(K)$, whose cells are in one-to-one correspondence with the simplices in the descending open stars of the strong critical vertices of $g$.
\end{theorem}

All three strong Morse matchings in Figure \ref{fig:C_6} collapse to a regular CW-core with two vertices and two edges between them, clearly preserving homotopy type. All matchings in Figure \ref{fig:triangles} collapse the complex to a single point as expected.

\begin{algorithm}[htb!]
\caption{A random strong Morse matching \cite{Fernandez2026}}
\label{alg:internal_core_computation}
\begin{algorithmic}
\State \verb"Input:" $X$, the face poset of a simplicial complex $K$
\State \verb"Output:" A strong Morse matching ${\mathcal M}$ on $K$
\\
    \State $S \gets \text{Shuffle}(\text{MinimalElements}(X))$ \Comment{{\footnotesize Randomized linear extension of vertices of $K$}}
    \State \verb"Initialize" ${\mathcal M}=\emptyset$
    \For{$v \in S$}
        \If{there exists $a \in \text{MinimalElements}(X) \setminus \{v\}$ such that $a$ dominates $v$}
            \ForAll{$\sigma \in \{\sigma \in X\colon v \leq \sigma\}$ such that $a \nleq \sigma$}
                \State ${\mathcal M} \gets {\mathcal M} \cup \{ (\sigma, \{a\}\cup\sigma) \}$
            \EndFor
            \State $X \gets X \setminus \{ \tau \in X \colon v \leq \tau \}$
        \EndIf
    \EndFor
    \State \Return ${\mathcal M}$
\end{algorithmic}
\end{algorithm}

\section{The complex $\sm(K)$}\label{sec:smk}

The complex of discrete Morse functions was first defined and studied by Chari and Joswig in \cite{ChariJoswig2005}. We review their construction, then define an analogous complex generated by strong discrete Morse matchings.

\begin{definition}[\cite{ChariJoswig2005}]
Given a finite abstract simplicial complex $K$, the \textit{complex of discrete Morse matchings} $\mathfrak{M}(K)$ is the simplicial complex with vertex set consisting of the set of edges of the Hasse diagram of $K$ and whose simplices are all subsets of edges of the Hasse diagram of $K$ that form Morse matchings. The subcomplex generated by all facets in ${\mathfrak M}(K)$ is called the {\em pure complex of Morse matchings} and is denoted by ${\mathfrak M}_{\textrm{pure}}(K)$.
\end{definition}

The complex ${\mathfrak M}(K)$ is not pure in general. An easy example is given by the cycle $C_n$ for $n\ge 6$. The facets in ${\mathfrak M}(C_n)$ have dimension $n-2$; these are the matchings containing $n-1$ vertex-edge pairs that go around the cycle leaving only one pair unmatched. For $n\ge 6$, it is possible to leave two edges unmatched in such a way that the matching cannot be extended to one of size $n-1$, creating a maximal simplex that is not a facet. 

Note that if ${\mathcal M}$ is a Morse matching on $K$, any subset of ${\mathcal M}$ is also a Morse matching so that $\mathfrak{M}(K)$ is in fact a simplicial complex. This same fact {\em does not} hold for strong discrete Morse matchings, so we must define our complex slightly differently. An example of this can be seen in Figure \ref{fig:C_6}. Viewing the leftmost image as a Morse matching, the subset consisting of the single pair $(1,01)$ is still a valid Morse matching on $C_6$ and gives a vertex in ${\mathfrak M}(C_6)$. However, the single pair $(1,01)$ is not itself a strong Morse matching since there is no vertex filtration that yields it (every strong matching on $C_6$ consists of either four, two, or zero pairs). With that in mind, we define $\sm(K)$ as the downward closure of all strong discrete Morse matchings to ensure that it is a simplicial complex.

\begin{definition}\label{smdef}
Let K be a simplicial complex. The \textit{complex of strong discrete Morse matchings} on K, denoted $\mathcal{SM}(K)$, is the subcomplex of ${\mathfrak M}(K)$ generated by all strong  Morse matchings on $K$. The subcomplex generated by all facets of maximal dimension in $\sm(K)$ is called the {\em pure complex of strong Morse matchings} and is denoted by $\smpure(K)$. 
\end{definition}

These strong matchings are generated by vertex filtrations as described in Section \ref{sec:strong}. Since only the order of vertices matters when generating the strong matching, we typically label the vertices of K from $0$ to $n-1$, where $n=|V(K)|$. Every vertex filtration on $K$ produces a strong matching with $k$ pairs for some $0\leq k\leq \dim\sm(K)+1$, represented by a $(k-1)$-simplex in $\sm(K)$. The vertices of this simplex are given by the various simplex pairs in $K$, which may or may not be strong matchings by themselves. For example, the strong Morse matchings shown in Figure \ref{fig:C_6} are each  represented in $\sm(C_6)$ by a $3-$simplex. If $k=0$, we have the \textit{empty strong Morse matching} on $K$. Following the convention of \cite{ChariJoswig2005}, the empty matching is not included in $\mathcal{SM}(K)$.

\section{Some auxiliary results}\label{sec:aux} When computing the homotopy types of various $\sm(K)$ we will need to make use of the following results. The first is a special case of results in \cite{WelkerZiegler1999}, but we will need only this version.

\begin{proposition}[\cite{ChariJoswig2005}, Proposition 3.4]\label{wedgeprop} 
    Let $\Delta$ be a simplicial complex and suppose that $A,B$ are subcomplexes such that the inclusions $A\cap B\hookrightarrow A$ and $A\cap B\hookrightarrow B$ are null-homotopic. Then $A\cup B\simeq A\vee B\vee \Sigma(A\cap B)$.
\end{proposition}

We will also make use of {\em generalized discrete Morse theory} as described in \cite{BauerEdelsbrunner2017}. An {\em interval} in a simplicial complex $K$ is a subset of simplices of the form $$[P,R] = \{Q\mid P\subseteq Q\subseteq R\}.$$  The interval is nonempty if and only if $P$ is a face of $R$. A partition $W$ of $K$ into intervals is called a {\em generalized discrete vector field}. Now suppose there is a function $f:K\to\zr$ satisfying $f(P)\le f(Q)$ whenever $P$ is a face of $Q$ with equality holding if and only if both $P$ and $Q$ belong to a common interval in $W$. Then $f$ is called a {\em generalized discrete Morse function} and $W$ its {\em generalized discrete gradient}. If an interval contains only one simplex (i.e., $P=R$), then we call the interval {\em singular} and the simplex a {\em critical simplex}. 

Note that for every generalized discrete gradient, there is a discrete gradient vector field that refines every non-singular interval $[P,R]$ into pairs. Indeed, choose an arbitrary vertex $x\in R\setminus P$ and partition $[P,R]$ into pairs $\{Q\setminus \{x\},Q\cup\{x\}\}$ for all $Q\in [P,R]$. This is called a {\em vertex refinement}. This refinement and the generalized gradient have the same critical simplices. 

\begin{theorem}[\cite{BauerEdelsbrunner2017}, Theorem 2.2]\label{gencollapse}
    Let $K$ be a simplicial complex with a generalized discrete gradient $V$, and let $K'\subset K$ be a subcomplex. If $K\setminus K'$ is a union of non-singular intervals in $V$, then $K\searrow K'$.
\end{theorem}

\section{The path $P_n$}\label{sec:path}
Let $P_n$ be the path with $n$ edges and $n+1$ vertices labeled $v_0,v_1,\dots,v_n$. The complex $\sm(P_n)$ has $n(n+1)$ vertices corresponding to the vertex-edge pairs $01,10,12,21,\dots,(n-1,n),(n,n-1)$; number these as $0,1,2,\dots,2n-1$ (this is the lexicographical ordering). Since the vertex ordering $0,1,\dots,n$ yields a strong Morse matching, the complexes $\sm(P_n)$ have dimension $n-1$. Moreover, choosing a vertex to label with $0$ determines a maximal matching, so there are exactly $n+1$ $(n-1)$-simplices in $\sm(P_n)$. The complex $\sm(P_n)$ is pure for $n\le 3$. 

\begin{proposition}\label{prop:pnpure}
    The complex $\smpure(P_n)$ is contractible for all $n\ge 2$.
\end{proposition}

\begin{proof}
    The complex $\smpure(P_n)$ has facets $[0,2,\dots,2n-2]$, $[0,2,\dots,2n-4,2n-1]$, $[0,2,\dots,2n-6,2n-3,2n-1]$, ..., $[0,3,5,\dots,2n-1]$, and $[1,3,5,\dots,2n-1]$. These form a chain of $(n-1)$-simplices, each simplex joined to the next along a common face, and hence this space is contractible. 
\end{proof}

\begin{proposition}\label{shortpaths}
We have the following homotopy types:
$$\sm(P_1)= \zs^0;\quad \sm(P_2)\simeq\ast;\quad \sm(P_3)\simeq\ast;\quad \sm(P_4)\simeq \zs^1.$$
\end{proposition}

\begin{proof}
    For $P_1$, we get two possible gradients corresponding to the pairings $01$ and $10$ and so $\sm(P_1)=\zs^0$. The complex $\sm(P_2)$ has four vertices $0,1,2,3$ and edges $[1,3],[0,3],[0,2]$; this is contractible. The complex $\sm(P_3)$ consists of four $2$-simplices: $[0,2,5],[0,3,5],[0,2,4],[1,3,5]$. This is also contractible.

    Finally, the complex $\sm(P_4)$ is shown in Figure \ref{p4diagram}. Its facets are  $[0, 2, 5, 7]$, $[0, 3, 5, 7]$, $[0, 2, 4, 7]$, $[0, 2, 4, 6]$, and $[1, 3, 5, 7]$, along with the edge $[1,6]$. This is clearly homotopic to $\zs^1$ as the subcomplex generated by the five tetrahedra is contractible and contains the boundary of the edge $[1,6]$.
\end{proof}

\begin{figure}
    \begin{tikzpicture}[scale=1]
  \coordinate (V0) at (112.5:3.5);
  \coordinate (V2) at (67.5:3.5);
  \coordinate (V4) at (22.5:3.5);
  \coordinate (V6) at (-22.5:3.5);
  \coordinate (V1) at (-67.5:3.5);
  \coordinate (V3) at (-112.5:3.5);
  \coordinate (V5) at (-157.5:3.5);
  \coordinate (V7) at (157.5:3.5);

  \fill[black!30] (V0) -- (V2) -- (V4) -- cycle;
  \fill[black!30] (V0) -- (V4) -- (V6) -- cycle;
  \fill[black!30] (V0) -- (V7) -- (V5) -- cycle;
  \fill[black!30] (V5) -- (V3) -- (V1) -- cycle;
  \fill[black!30] (V0) -- (V2) -- (V5) -- cycle;
  \fill[black!30] (V0) -- (V3) -- (V5) -- cycle;
  \fill[black!30] (V1) -- (V3) -- (V7) -- cycle;
  \fill[black!30] (V0) -- (V4) -- (V7) -- cycle;

  \draw[thick] (V0) -- (V2) -- (V4) -- (V6) -- (V1) -- (V3) -- (V5) -- (V7) -- cycle;

  \draw[thick] (V0) -- (V4);
  \draw[thick] (V0) -- (V6);
  \draw[thick] (V0) -- (V3);
  \draw[thick] (V0) -- (V5);

  \draw[thick] (V7) -- (V4);
  \draw[thick] (V7) -- (V1);
  \draw[thick] (V7) -- (V3);

  \draw[thick] (V5) -- (V2);
  \draw[thick] (V5) -- (V1);

  \foreach \v in {V0, V1, V2, V3, V4, V5, V6, V7} {
    \fill (\v) circle (3pt);
  }

  \node[above left=2pt] at (V0) {\large\textbf{0}};
  \node[above=3pt] at (V2) {\large\textbf{2}};
  \node[right=4pt] at (V4) {\large\textbf{4}};
  \node[right=4pt] at (V6) {\large\textbf{6}};
  \node[below right=2pt] at (V1) {\large\textbf{1}};
  \node[below left=2pt] at (V3) {\large\textbf{3}};
  \node[left=4pt] at (V5) {\large\textbf{5}};
  \node[left=4pt] at (V7) {\large\textbf{7}};

\end{tikzpicture}
\caption{\label{p4diagram} The complex $\sm(P_4)$.}
\end{figure}
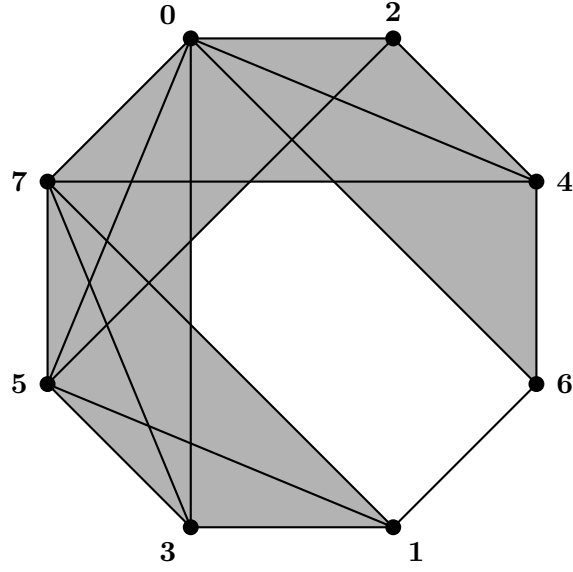

Larger values of $n$ require a deeper analysis. First note that the complexes are not pure. We have already seen that the complex $\sm(P_4)$ contains the maximal simplex $[1,6]$ corresponding to the following strong Morse matching:

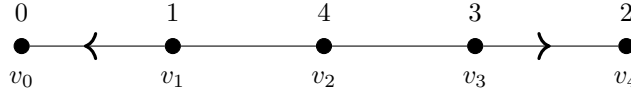
\begin{center}
\begin{tikzpicture}

\node[circle, draw, fill=black, inner sep=2pt] (v0) at (0,0) {};
\node[circle, draw, fill=black, inner sep=2pt] (v1) at (2,0) {};
\node[circle, draw, fill=black, inner sep=2pt] (v2) at (4,0) {};
\node[circle, draw, fill=black, inner sep=2pt] (v3) at (6,0) {};
\node[circle, draw, fill=black, inner sep=2pt] (v4) at (8,0) {};

\node[below=6pt] at (v0) {$v_0$};
\node[below=6pt] at (v1) {$v_1$};
\node[below=6pt] at (v2) {$v_2$};
\node[below=6pt] at (v3) {$v_3$};
\node[below=6pt] at (v4) {$v_4$};

\node[above=6pt] at (v0) {$0$};
\node[above=6pt] at (v1) {$1$};
\node[above=6pt] at (v2) {$4$};
\node[above=6pt] at (v3) {$3$};
\node[above=6pt] at (v4) {$2$};

\draw[postaction={decorate, decoration={markings,
    mark=at position 0.5 with {\arrow[thick,scale=1.5]{<}}}}]
    (v0) -- (v1);

\draw (v1) -- (v2) -- (v3);

\draw[postaction={decorate, decoration={markings,
    mark=at position 0.5 with {\arrow[thick,scale=1.5]{>}}}}]
    (v3) -- (v4);

\end{tikzpicture}
\end{center}
\noindent This edge is not contained in any facet of $\sm(P_4)$. Note, however, that this is a face of a $2$-simplex in the full Morse complex ${\mathfrak M}(P_4)$, providing an example to show that the inclusion $\sm(P_4)\subset {\mathfrak M}(P_4)$ is strict.

We now present an alternate description of the facets of $\sm(P_n)$.

\begin{definition}
A \emph{direction string} is a string $d = d_0 d_1 \cdots d_{n-1}$ of length
$n$ over the alphabet $\{{\tt >},\, {\tt <},\, {\tt -}\}$, where ${\tt >}$ at
position $k$ means the rightward edge $(k, k+1)$ is active (contributing vertex
$2k$), ${\tt <}$ at position $k$ means the leftward edge $(k-1, k)$ is active
(contributing vertex $2k+1$), and ${\tt -}$ at position $k$ means no vertex is
contributed at that position.
\end{definition}

\begin{definition}\label{def:valid}
A direction string is \emph{valid} if:
\begin{enumerate}
\item No consecutive ${\tt <>}$ pair appears (a vertex cannot
      support a left arrow and a right arrow simultaneously).
\item Every maximal run of ${\tt -}$ characters has even length.
\end{enumerate}
\end{definition}

Every valid direction string with at least one non-dash character corresponds to a unique simplex of $\sm(P_n)$, and every simplex arises this way. The $(n-1)$-facets correspond to the $n+1$ strings $({\tt >})^j({\tt <})^{n-j}$ for $j=0,1,\dots,n$, obtained by choosing vertex $j$ to be labeled with $0$. The prohibition on odd-length runs of ${\tt -}$ characters arises from the following observation. The string ${\tt <}{-}{\tt >}$ corresponds to four consecutive vertices. Say they are labeled $a,b,c,d$ from left to right. Then we would have $a<b$, $d<c$, and either $b>c$ or $c>b$; that is, the lower link of either $b$ or $c$ would be $S^0$, which is not a cone. Longer runs of odd length can be reduced to this by filling in all but one ${\tt -}$ with an appropriate run of ${\tt <}$ or ${\tt >}$ characters.

\begin{definition}\label{stuck}
A ${\tt --}$ pair at positions $k$, $k+1$ in a valid direction string is \emph{stuck} if
\[
d_{k-1} = {\tt <} \quad \text{and} \quad d_{k+2} = {\tt >}.
\]
The ${\tt <}$ to the left prevents replacing $d_k$ with ${\tt >}$ (that would create a forbidden ${\tt <>}$ at positions $k-1, k$), and the ${\tt >}$ to the right prevents replacing $d_{k+1}$ with ${\tt <}$ (that would create a forbidden ${\tt <>}$ at positions $k+1, k+2$). Neither dash can be made active without violating validity, so the pair genuinely cannot be extended to a larger
matching.
\end{definition}

Note that a simplex is \emph{nonpure-maximal} (maximal among all simplices of dimension less than $n-1$) if and only if:
\begin{itemize}
\item its direction string contains at least one ${\tt --}$ pair (equivalently, it is not a pure facet), and
\item every ${\tt --}$ pair in the string is stuck.
\end{itemize}
This gives an $O(n)$ test per string. The key excluded cases are direction strings with no dashes (those are the $n+1$ pure facets) and strings where some ${\tt --}$ pair is not stuck (those can be extended to a strictly larger string by replacing the unstuck dashes with active edges).

A notable family of nonpure-maximal facets is the \emph{repeated stuck pattern}: the string $({\tt <}{-}{-}{\tt >})^k$ for $n = 4k$, which gives a single maximal simplex of dimension $n/2-1$. For $n \equiv 0 \pmod{4}$ this is the unique facet of minimum dimension among the nonpure maximal simplices.

A final observation about this characterization of simplices in $\sm(P_n)$: no pair of vertices of the form $\{i,i+3\}$ for $i$ odd can occur in a simplex. This would correspond to the string ${\tt <}{\tt -}{\tt >}$, which has an odd number of ${\tt -}$ characters. Pairs $\{j,j+3\}$ {\em are} possible for $j$ even, as these correspond to the string ${\tt >}{\tt <}$.

\begin{theorem}\label{pathhtyptype}
    For $n\ge 4$ we have the following homotopy types:
    $$\sm(P_n) \simeq \begin{cases}
        \zs^{2k-1} & n=4k \\
        \zs^{2k} & n=4k+1 \\
        \ast & n=4k+2,4k+3.
    \end{cases}$$
\end{theorem}

\begin{proof}
    Note that $\sm(P_n) = \text{St}(0)\cup \text{St}(1)$, both of which are contractible. Further, write $\text{St}(1)=\text{St}(13)\cup \text{St}(16)$. Then $\sm(P_n)\simeq \ast\vee\ast\vee \Sigma(\text{St}(0)\cap(\text{St}(13)\cup \text{St}(16)))$. Now consider this intersection and write it as $$(\text{St}(0)\cap \text{St}(13))\cup (\text{St}(0)\cap \text{St}(16)) = A\cup B.$$ 

    We first make the following observations. Any element of $\text{St}(16)$ cannot contain vertices $2,3,4,5,7$:  the vertices $1$ and $6$ bookend a stuck pair so that $3,4,5$ are eliminated, vertex $2$ cannot occur with vertex $1$, and vertex $7$ cannot occur with vertex $6$. Any element of $\text{St}(13)$ cannot contain vertex $0$, nor can it contain vertices $2$ or $6$ (pairs $i,i+3$ are forbidden for $i$ odd). Any element of $\text{St}(0)$ cannot contain vertex $1$.
    
    Note that every facet of $A$ excludes vertices $0$, $1$, and $2$ and contains the vertex $3$, so that $A$ is a cone with apex $3$ and is therefore contractible. Similarly, every facet of $B$ excludes vertices $0$, $1$, $2$, $3$, $4$, and $5$ and contains vertex $6$ and is thus contractible. It follows that $$A\cup B \simeq \ast\vee\ast \vee\Sigma(A\cap B)$$ and hence $$\sm(P_n)\simeq \Sigma^2(A\cap B).$$
    
    Now note that $$A\cap B = \text{St}(0)\cap \text{St}(13)\cap \text{St}(16).$$ Consider the complex $X=\sm(P_n)\setminus\{0,1,2,3,4,5,6,7\}$. This is isomorphic to the strong Morse complex on $n-4$ vertices; that is, it is isomorphic $\sm(P_{n-4})$. Observe that $A\cap B$ is contained in this complex. Indeed, any element of $\text{St}(16)$ cannot contain vertices $2,3,4,5,7$; any element of $\text{St}(13)$ cannot contain vertices $0$, $2$, or $6$; and any element of $\text{St}(0)$ cannot contain vertex $1$. Thus an element of the intersection lies in $X$. Now note that any facet of $X$ can be extended to a facet of all three terms of the intersection: if a facet has initial vertex $8$, then it can be extended by $[0,2,4,6]$, $[1,3]$, and $[1,6]$, and if a facet has initial vertex $9$, then it can be extended by $[0,3,5,7]$, $[1,3,5,7]$, and $[1,6]$. We conclude that $X$ equals $A\cap B$ and hence $$\sm(P_n)\simeq \Sigma^2\sm(P_{n-4}).$$ The result now follows from Proposition \ref{shortpaths}.
\end{proof}

\begin{remark}
    Contrast this computation of the homotopy type of $\sm(P_n)$ with that of ${\mathfrak M}(P_n)$ \cite[Prop. 4.6]{Kozlov1999}. Here the result depends on the residue of $n$ modulo $4$, while the corresponding modulus for ${\mathfrak M}(P_n)$ is $3$: $${\mathfrak M}(P_n) \simeq \begin{cases}
        \zs^{2k-1} & n=3k \\
        \zs^{2k} & n=3k+1 \\
        \ast & n=3k+2.
    \end{cases}$$
\end{remark}

\section{The star graph $S_n$}\label{sec:star}
Denote the complete bipartite graph $K_{1,n}$ by $S_n$. This is often called the {\em star graph}. The $n=10$ case is shown in Figure \ref{fig:star}.

\begin{figure}
   
\begin{tikzpicture}
  \tikzset{
    vertex/.style={circle, fill=black, inner sep=1.8pt},
    edge/.style={thick, black}
  }
  \node[vertex, label={[label distance=6pt]below:$v_0$}] (v0) at (0,0) {};
  \foreach \i in {1,...,10} {
    \pgfmathsetmacro{\ang}{(\i-1)*36}
    \node[vertex, label={\ang:$v_{\i}$}] (v\i) at (\ang:2) {};
    \draw[edge] (v0) -- (v\i);
  }
\end{tikzpicture}

    \caption{\label{fig:star} The star graph $S_{10}$}
\end{figure}
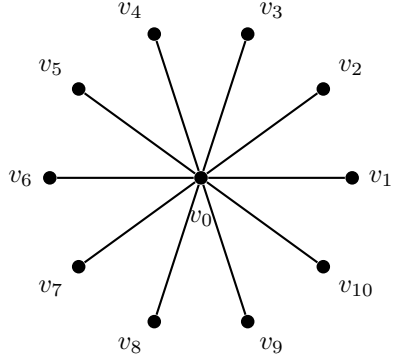

\begin{theorem}\label{startheorem}
    For all $n\ge 2$, $\sm(S_n)$ is contractible.
\end{theorem}

\begin{proof}
    Note that $S_1=P_1$ and hence $\sm(S_1)\simeq \zs^0$. Now assume $n\ge 2$. The complex $\sm(P_n)$ has $2n$ vertices corresponding to the $2n$ possible vertex-edge pairs in $S_n$. Order these lexicographically--$01,02,\dots ,0n,10,20,\dots ,n0$--and then number them $1,2,\dots ,2n$.  
    
    If we label vertex $v_0$ with $0$, then for any labeling of the other $n$ vertices, we have $\textrm{lk}^\downarrow(v_i) = v_0$. We therefore obtain the maximal matching in which every vertex $v_i$ is paired with the edge joining it to $v_0$ and $v_0$ is left critical. The corresponding facet of $\sm(S_n)$ is $$[n+1,n+2,\dots ,2n].$$ 

    Now suppose that $v_i$ is labeled with $0$ and $v_0$ is labeled with $1$. For any ordering of the other $n-1$ vertices the lower link will consist of $v_0$ and so we obtain the maximal matchings $$[i,n+1,n+2,\dots ,\widehat{n+i},\dots ,2n]$$ for $1\le i\le n$.
    
    Finally consider labelings where $v_0$ is assigned value $k\ge 2$. The $k$ vertices labeled with values $0,\dots ,k-1$ are unpaired, and when $v_0$ gets its label, these vertices form $\textrm{lk}^\downarrow(v_0)$ and hence $v_0$ is not paired with any edge. If $v_j$ is any remaining vertex, then $\textrm{lk}^\downarrow(v_j) = \{v_0\}$. If $v_{i_1},\dots ,v_{i_{n-k}}$ are these vertices, then we obtain the matching $$[n+i_1,n+i_2,\dots ,n+{i_{n-k}}],$$ which is a face of the facet $[n+1,n+2,\dots ,2n]$ above. It follows that $\sm(S_n)$ is pure with the $n+1$ $(n-1)$-simplices listed above as facets.

    The complex $\sm(S_n)$ is shellable, with shelling order $$[n+1,n+2,\dots ,2n],[1,n+2,\dots ,2n],[2,n+1,n+3,\dots,2n],\dots ,[n,n+1,\dots ,2n-1].$$ There are no generating simplices since for each $i$, $F_i\cap (\bigcup_{j=1}^{i-1}F_j)$ is not the entire boundary of $F_i$ (the initial vertex of each $F_i$ does not appear in any of its predecessors) and so $\sm(S_n)$ is contractible.
\end{proof}

\begin{remark}
    Note that we have an equality $\sm(S_n)={\mathfrak M}(S_n)$. Indeed, any Morse matching on $S_n$ has at most one vertex-edge pair of the form $0i$, and then any remaining vertex may be paired with the edge joining it to $v_0$. Thus the facets of ${\mathfrak M}(S_n)$ all lie in $\sm(S_n)$. Proposition 3.5 of \cite{DonovonLinScoville2023} implies that ${\mathfrak M}(S_n)$ is contractible; Theorem \ref{startheorem} provides an alternate proof of this. 
\end{remark}

\section{The cycle graph $C_n$}\label{sec:cycle}
Let $C_n$, $n\ge 3$, be the $n$-cycle with $n$ vertices and $n$ edges. We adopt the following notation from \cite{ChariJoswig2005}. The vertices of $C_n$ are denoted $v_0,v_1,\dots,v_{n-1}$ in clockwise order, with edges $[v_i,v_{i+1}]$ and indices taken modulo $n$. The $2n$ vertices of $\mathcal{SM}(C_n)$ are identified with numbers $0,1,2\dots,2n-1$ such that the vertex $2i$ corresponds to the pair $(v_i,[v_i,v_{i+1}])$ and $2i+1$ corresponds to the pair $(v_{i+1},[v_i,v_{i+1}])$.

\begin{remark}
    Any strong matching on $C_n$ cannot have an odd number of consecutive unpaired edges. This is by the same argument presented in Definition \ref{def:valid}, since $C_n\setminus v$ for any vertex $v$ is exactly $P_{n-2}$.
\end{remark}

\begin{proposition}
    The complex $\mathcal{SM}(C_n)$ is $(n-3)-$dimensional, $n\geq 3.$
\end{proposition}

\begin{proof}
    Any vertex filtration on $C_n$ ends with some $v_i$ closing the cycle. The vertex $v_i$ is not descending dominated, since its lower link in $C_n$ is the disjoint union of $v_{i-1}$ and $v_{i+1}$. So $v_i,[v_{i-1},v_i],$ and $[v_i,v_{i+1}]$ are not included in the resulting strong Morse matching. The result is, at most, a strong $(n-2)-$matching, which can be achieved by numbering the vertices by their indices. A strong $(n-2)-$matching on $C_n$ is represented in $\mathcal{SM}(C_n)$ by a $(n-3)-$simplex.
\end{proof}

\begin{proposition}
    The number of $(n-3)-$simplices in $\mathcal{SM}(C_n)$ is $n(n-1)$, $n\geq 3$.
\end{proposition} 

\begin{proof}
    Any $(n-3)-$simplex is the result of a strong Morse matching on $C_n$ that has a single gap of two unpaired edges. Let $v_{i}$ be the vertex whose two edges are unpaired; all other vertices and edges are paired. Then $C_n\setminus v_{i}=P_{n-2}$, and $\sm(P_{n-2})$ has exactly $n-1$ simplices of maximal dimension $n-3$. The symmetry of $C_n$ results in $n(n-1)$ $(n-3)-$dimensional simplices in $\sm(C_n)$, the $n$ copies  corresponding to the $n$ possibilities of which vertex goes unpaired in $C_n$.
\end{proof}

\begin{theorem} The complex $\mathcal{SM}(C_n)$ is pure for $n\le 7$.
\end{theorem}

\begin{proof}
    For $n\ge 8$ we obtain maximal simplices of dimension less than $n-3$ because of the existence of two or more stuck pairs. For $n\le 7$, there can be at most one stuck pair in a maximal matching; that is, all maximal simplices have dimension $n-3$.
\end{proof}

\begin{proposition}
We have the following homotopy types:
\begin{gather*}
      \sm(C_3)\simeq\bigvee_5 \zs^0;\quad \sm(C_4)\simeq\bigvee_5 \zs^1;\quad \sm(C_5)\simeq \zs^1.
\end{gather*}
\end{proposition}

\begin{proof}
     There are six possible strong 1-matchings on $C_3$, and as $\sm(C_3)$ is $0$-dimensional, we have $\sm(C_3)\simeq\bigvee_5\zs^0$.

    Note that $\sm(C_4)=A\cup B$ where $$A=\{[0,2],[2,4],[4,6],[0,6],[1,6],[0,3],[2,5],[4,7]\}$$ and $$B=\{[1,3],[3,5],[5,7],[1,7]\}.$$ One checks easily that $A\simeq\zs^1\simeq B$ and $A\cap B=\{1,3,5,7\}\simeq\bigvee_3\zs^0$. By Proposition \ref{wedgeprop} we have $$\sm(C_4)\simeq A\vee B\vee\Sigma(A\cap B)\simeq\bigvee_5\zs^1$$

    To foreshadow the calculations for $n\ge 6$, we view $\sm(C_5)$ as a union of two spaces: $\text{St}(0)\cup\text{St}(6)$ and $\text{St}(1)\cup\text{St}(7)$. We construct an explicit discrete gradient on $\sm(C_5)$; the pairings are as follows.
    
\begin{center}
\setlength{\tabcolsep}{10pt} 
\begin{tabular}{ c c c c }
 $[7]<[7,9]$ & $[0,8]<[0,3,8]$ & $[1,3]<[1,3,8]$ & $[3,7]<[3,5,7]$ \\ 
 $[0]<[0,2]$ & $[4,6]<[4,6,9]$ & $[1,5]<[1,3,5]$ & $[3,9]<[1,3,9]$ \\  
 $[0,4]<[0,2,4]$ & $[0,6]<[0,6,8]$ & $[1,7]<[1,7,9]$ & $[4,7]<[4,7,9]$ \\
 $[0,5]<[0,2,5]$ & $[2,6]<[2,4,6]$ & $[2,7]<[2,4,7]$ & $[5,7]<[2,5,7]$ \\
 $[0,3]<[0,3,5]$ & $[4,8]<[4,6,8]$ & $[2,8]<[0,2,8]$ & $[5,9]<[5,7,9]$
\end{tabular}
\end{center}

\noindent The cells left critical are $[1,6,8],[1,6,9],[2,4],[2,5],[3,5],[3,8],[4,9],[1],[2],[3],$\\$[4],[5],[6],[8],[9]$. This is precisely the intersection $(\text{St}(0)\cup\text{St}(6))\cap (\text{St}(1)\cup\text{St}(7))$. This space has the  homotopy type of $\zs^1$, as shown in Figure \ref{c5diagram}. 
\end{proof}

Recall that $\sm(C_n)=\smpure(C_n)$ for $n\leq 7$. There is no discernible pattern in the homotopy types of $\smpure(C_n)$ for small $n$. However, for $n\geq6$, a pattern emerges. 

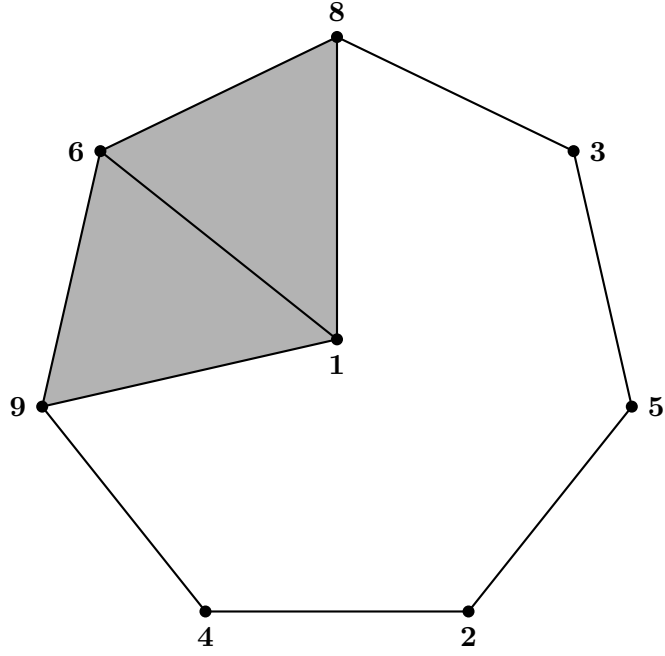
\begin{figure}
\begin{tikzpicture}[
  vtx/.style={circle, fill=black, inner sep=1.6pt},
  lbl/.style={font=\small}, scale=1
]

\coordinate (v1) at (0,0);
\coordinate (v6) at (-3.13,2.49);
\coordinate (v8) at (0,4);
\coordinate (v9) at (-3.90,-0.89);
\coordinate (v3) at (3.13,2.49);
\coordinate (v5) at (3.90,-0.89);
\coordinate (v2) at (1.74,-3.60);
\coordinate (v4) at (-1.74,-3.60);

\fill[black!30]  (v1) -- (v6) -- (v8) -- cycle;
\fill[black!30] (v1) -- (v6) -- (v9) -- cycle;

\draw[thick] (v1) -- (v6);
\draw[thick] (v1) -- (v8);
\draw[thick] (v6) -- (v8);
\draw[thick] (v1) -- (v9);
\draw[thick] (v6) -- (v9);

\draw[thick] (v2) -- (v4);
\draw[thick] (v2) -- (v5);
\draw[thick] (v3) -- (v5);
\draw[thick] (v3) -- (v8);
\draw[thick] (v4) -- (v9);

\node[vtx, label={below:\large\textbf 1}] at (v1) {};
\node[vtx, label={left:\large\textbf 6}]  at (v6) {};
\node[vtx, label={above:\large\textbf 8}] at (v8) {};
\node[vtx, label={left:\large\textbf 9}]  at (v9) {};
\node[vtx, label={right:\large\textbf 3}] at (v3) {};
\node[vtx, label={right:\large\textbf 5}] at (v5) {};
\node[vtx, label={below:\large\textbf 2}] at (v2) {};
\node[vtx, label={below:\large\textbf 4}] at (v4) {};

\end{tikzpicture}
\caption{\label{c5diagram} The critical simplices of the gradient on $\sm(C_5)$.}
\end{figure}

\begin{theorem}\label{thm:cnpure} The homotopy type of $\smpure(C_n)$, for $n\geq 6$, is
\begin{equation*}
\smpure(C_n)\simeq
    \begin{cases}
        \mathbb{S}^{2}\vee\mathbb{S}^{2k-2}\vee\mathbb{S}^{2k-2}\vee\mathbb{S}^{2k-2}\vee\mathbb{S}^{2k-2} & \text{if } n=3k\\
        \mathbb{S}^{2}\vee\mathbb{S}^{2k-1}\vee\mathbb{S}^{2k-1} & \text{if } n=3k+1, 3k+2\\
    \end{cases}
\end{equation*}
\end{theorem}

\begin{proof}
We first make an observation. The space $\smpure(C_n)$ is the union of $\text{St}(0)$, $\text{St}(1)$, $\text{St}(6)$, and $\text{St}(7)$. Indeed, if a facet does not lie in $\text{St}(0) \cup \text{St}(1)$, then the edge $[v_0,v_1] \in C_n$ is unpaired. It could be that the edge $[v_1,v_2]$ is also unpaired, but then the edge $[v_2,v_3]$ must be paired; that is, the corresponding facet lies in $\text{St}(6)\cup\text{St}(7)$. If $[v_1,v_2]$ is paired, then since $[v_0,v_1]$ is unpaired, it must be the case that $[v_2,v_3]$ is also paired (since a single paired edge cannot be flanked by two unpaired edges); that is, the facet lies in $\text{St}(6)\cup\text{St}(7)$.

Note also that $\text{St}(0)\cup\text{St}(6) \cong \text{St}(1)\cup\text{St}(7)$. So we must compute the homotopy types of $\text{St}(0)\cup\text{St}(6)$ and the intersection
$$(\text{St}(0)\cup\text{St}(6)) \cap (\text{St}(1)\cup\text{St}(7)).$$

\begin{lemma}\label{lem:intersection}
    The space $(\text{St}(0)\cup\text{St}(6)) \cap (\text{St}(1)\cup\text{St}(7))$ has the homotopy type of $\zs^1$.
\end{lemma}

\begin{proof}
    Denote this intersection by $Z$. The facets of $Z$ fall into four classes, determined by where stuck pairs appear, giving a gap of $5$ in the vertices. The groups are
    \begin{eqnarray*}
        A & = & \{[0,2,4,7,\dots,2n-5],[0,2,4,7,\dots,2n-7,2n-2],\dots, \\
        &  & \hskip 6pt [0,2,4,7,12,\dots,2n-2],[0,3,5,\dots,2n-5], \\
        &  & \hskip 6pt [0,3,5,\dots,2n-7,2n-2],\dots,[0,3,5,7,12,\dots,2n-2]\} \\
        B & = & \{[1,6,8,\dots,2n-2],[1,6,8,\dots,2n-4,2n-1],\dots, \\
        &  &  \hskip 6pt [1,6,8,11,\dots,2n-1], [1,6,9,11,\dots,2n-1]\} \\
        C & = & \{[2,4,9,11,\dots,2n-3],[4,9,11,\dots,2n-3,2n-1]\} \\
        D & = & \{[3,5,10,12,\dots, 2n-2],[3,8,10,12,\dots,2n-2]\}
    \end{eqnarray*}

    Note that $A$,$B$,$C$, and $D$ are all contractible: $A$ and $B$ are cones and $C$ and $D$ both consist of a pair of simplices joined along a common face. Note also that $C\cap D=\emptyset$. Our strategy is to consider various unions and intersections so that we may apply Proposition \ref{wedgeprop}. For now, assume the various inclusions of intersections into the corresponding spaces are null-homotopic. This will be justified as we move along. First, we have
    \begin{eqnarray*}
        Z & = & (A\cup B\cup C) \cup D \\
          & \simeq & (A\cup B\cup C) \vee \ast \vee \Sigma((A\cup B \cup C)\cap D).
    \end{eqnarray*}
    Observe that 
    \begin{eqnarray*}
        (A\cup B\cup C)\cap D & = & (A\cap D)\cup (B\cap D)\cup (C\cap D) \\
          & = & \{[3,5,12,\dots,2n-2],[8,10,12,\dots,2n-2]\}.
    \end{eqnarray*}
    For $n\ge 7$, this space is contractible as it is two simplices joined along a common face. We conclude that for $n\ge 7$, $Z\simeq A\cup B\cup C$. When $n=6$, this space is the disjoint union of the two intervals $[3,5]$, $[8,10]$, and so in this case we have $$Z\simeq (A\cup B\cup C) \vee \Sigma \zs^0 = (A\cup B\cup C)\vee \zs^1.$$ This last assertion requires that $A\cup B\cup C$ be connected, which we will show below.

    Now, 
    \begin{eqnarray*}
        A\cup B\cup C & = & (A\cup C)\cup B \\
           & \simeq & (A\cup C) \vee \ast \vee \Sigma((A\cup C)\cap B).
    \end{eqnarray*}
    We have, for $n\ge 7$,
    \begin{eqnarray*}
        (A\cup C)\cap B & = & \{[12,14,\dots,2n-2],[9,11,\dots,2n-1]\} \\
           & \simeq & \zs^0.
    \end{eqnarray*}
    Thus, for $n\ge 7$,
    \begin{eqnarray*}
        A\cup B\cup C & \simeq & (A\cup C)\vee \Sigma\zs^0\\
         & = & (A\cup C)\vee \zs^1.
    \end{eqnarray*}
    Also, for $n\ge 7$, we have
    $$A\cup C \simeq  \ast\vee\ast\vee \Sigma(A\cap C),$$
    where
    $$A\cap C = \{[2,4,9,11,\dots,2n-5]\} \simeq\ast.$$
    Thus, for $n\ge 7$ we have $Z\simeq \zs^1$.

    When $n=6$, we have $$(A\cup C)\cap B = \{[9,11]\} \simeq \ast,$$ so that $A\cup B \cup C \simeq A\cup C$. Finally, note  that $$A\cup C = \{[0,3,5,7],[0,2,5,7],[0,2,4,7],[2,4,9],[4,9,11]\},$$  and this space is contractible (shellable, even, in the order given). It follows that for $n=6$ we also have $Z\simeq (A\cup B\cup C) \vee \zs^1 \simeq \zs^1$.
\end{proof}

Next, let $E_n$  be the subcomplex of $\smpure(C_n)$ generated by the even vertices: $\{0,2,4,\dots,2n-2\}$. Similarly, $O_n$ is the subcomplex generated by $\{1,3,5,\dots,2n-1\}$. Note that $E_n\subset \text{St}(0)\cup\text{St}(6)$, $O_n\subset \text{St}(1)\cup\text{St}(7)$, and $E_n\cong O_n$. 

\begin{lemma}\label{lem:endefretract} The space $\text{St}(0)\cup \text{St}(6)$ collapses to $E_n$.
\end{lemma}

\begin{proof}
    We make use of generalized discrete Morse theory (Section \ref{sec:aux}). Note that each facet of $(\text{St}(0)\cup\text{St}(6))\setminus E_n$ contains at least one odd vertex.  For each simplex $\tau\in(\text{St}(0)\cup\text{St}(6))\setminus E_n$ containing at least one even vertex, let $\alpha_\tau\subset\tau$ be the face generated by the odd vertices of $\tau$. The intervals $[\alpha_\tau,\tau]$ form a partition of $(\text{St}(0)\cup\text{St}(6))\setminus E_n$. Indeed, any simplex in this set is a face of some simplex containing at least one even vertex and therefore lies in some interval $[\alpha_\tau,\tau]$. But the definition of $\alpha_\tau$ implies that the intervals are disjoint: if some $\sigma$ were to lie in two such intervals corresponding to $\tau_1$ and $\tau_2$, then the sets of odd vertices would both be contained among the vertices of $\sigma$ and therefore in both $\tau_1$ and $\tau_2$, which forces $\tau_1=\tau_2$. Now partition $E_n$ by letting each simplex be a singleton.  Then this partition of $(\text{St}(0)\cup\text{St}(6))$ determines a generalized discrete gradient vector field. By Theorem \ref{gencollapse}, we conclude that $\text{St}(0)\cup\text{St}(6)\searrow E_n$.
\end{proof}

\begin{remark}
    Define a map $f:(\text{St}(0)\cup\text{St}(6))\setminus E_n\to\zr$ as follows. On vertices, we set $f(2i)=0$ and $f(2i+1)=2i+1$. Then for a simplex $\sigma = [i_0,\dots ,i_k]$ we set $f(\sigma)=\sum_jf(i_j)$. Then each simplex $\sigma$ in the interval $[\alpha_\tau,\tau]$ satisfies $f(\sigma)=f(\alpha_\tau)$ and for all simplices $\mu\subset \nu$ we have $f(\mu)\le f(\nu)$ with equality only if $\mu$ and $\nu$ lie in some interval $[\alpha_\tau,\tau]$. Setting $f(\sigma)=\dim\sigma - N$ for some integer $N>n-3 = \dim E_n$ makes every simplex in $E_n$ critical for $f$. Thus $f$ is a generalized discrete Morse function for the generalized gradient defined above.
\end{remark}

Similarly, $\text{St}(1)\cup\text{St}(7)\searrow O_n$.
We now compute the homotopy type of $E_n$. Since $O_n\cong E_n$, this will also determine its homotopy type.

Let $A_n=\text{St}_{E_n}(0)$ and $B_n=E_n\setminus 0$ so that $E_n=A_n\cup B_n$, and set $A_n\cap B_n=D_n$. Note that $A_n$ is contractible as a closed star, so Proposition \ref{wedgeprop} gives 
\begin{eqnarray*}
    E_n &\simeq & A_n\vee B_n\vee\Sigma(A_n\cap B_n)\\
    &\simeq & B_n\vee\Sigma(D_n)
\end{eqnarray*}
So we must compute the homotopy types of $B_n$ and $D_n$.

\begin{lemma}\label{lem:bn}
    For $n\geq 9$, $B_n\simeq\Sigma^2(B_{n-3})$.
\end{lemma}

\begin{proof}
    There are $(n-2)$ maximal simplices of $B_n$. Of dimension $(n-3)$ we have $$\sigma_n=[2,4,\dots,2n-4]~\text{and}~\tau_n=[4,6,\dots,2n-2],$$ and of dimension $(n-4)$ we have
    \begin{gather*}
        [2,8,10,\dots,2n-2], [2,4,10,12,\dots,2n-2],\\ [2,4,6,12,14,\dots,2n-2],\dots, [2,4,\dots,2n-8,2n-2]
    \end{gather*}
    This gives a total of $(n-2)$ maximal simplices. Notice that $B_n$ is a union of two contractible pieces:
        $$B_n=\text{St}_{B_n}(2n-2)\cup\sigma_n$$
Denote the intersection $\text{St}_{B_n}(2n-2)\cap \sigma_n$ by $B_n'$. Then we have
    \begin{eqnarray*}
    B'_n &= & \text{St}_{B_n}(2n-2)\cap\sigma_n \\ 
    & = & \{[4,\dots,2n-4],[2,8,\dots,2n-4],[2,4,10,\dots,2n-4],\\
    &  & \quad [2,4,6,12,\dots,2n-4],\dots,[2,4,\dots,2n-10,2n-4],[2,\dots,2n-8]\}\\
    &= &\text{St}_{B'_n}(2n-4)\cup[2,\dots,2n-8]
    \end{eqnarray*}
The intersection $B_n''$ of these two pieces is then
    \begin{eqnarray*}
    B''_n &= & \text{St}_{B'_n}(2n-4)\cap[2,\dots,2n-8] \\
     & = & \{[4,\dots,2n-8],[2,8,10,\dots,2n-8] [2,4,10,\dots,2n-8],\\
    & & \quad\dots, [2,\dots,2n-14,2n-8],[2,4,\dots,2n-10]\}\\
    \end{eqnarray*}
But now note that $\sigma_{n-3}=[2,\dots,2n-10]$ and $\tau_{n-3} = [4,\dots,2n-8]$ so that
\begin{eqnarray*}
B''_n & = & \{\sigma_{n-3},\tau_{n-3},[2,8,10,\dots,2n-8],[2,4,10,\dots,2n-8], \\
  & & \quad\dots,[2,\dots,2n-14,2n-8]\},
\end{eqnarray*}
which is exactly $B_{n-3}$. Therefore, by Proposition \ref{wedgeprop},
    \begin{eqnarray*}
    B_n& \simeq &\text{St}_{B_n}(2n-2)\vee\sigma_n\vee\Sigma^2(B_{n-3})\\
    &\simeq & \ast\vee\ast\vee\Sigma^2(B_{n-3})\\
    &\simeq & \Sigma^2(B_{n-3}).
    \end{eqnarray*}
\end{proof}

\begin{lemma}\label{lem:dn}
    For $n\geq 9$, $D_n\simeq\Sigma^2(D_{n-3})$.
\end{lemma}

\begin{proof}
We employ the same strategy as Lemma \ref{lem:bn}. There are $(n-2)$ maximal simplices of $D_n$, all of dimension $(n-2)$.
    \begin{eqnarray*}
        D_n & = & \{[2,4,\dots,2n-6],[6,8,\dots,2n-2],[2,8,10,\dots,2n-2], \\
        & & \quad [2,4,10,\dots,2n-2],\dots,[2,\dots,2n-8,2n-2]\}\\
        &= &\text{St}_{D_n}(2n-2)\cup[2,4,\dots,2n-6]
    \end{eqnarray*}
The intersection $D_n'$ of these two pieces is
    \begin{eqnarray*}
        D'_n & = &\text{St}_{D_n}(2n-2)\cap[2,4,\dots,2n-6]\\
        & = &\{[2,4,\dots,2n-8],[6,\dots,2n-6],[2,8,\dots,2n-6],\\
        & & \quad[2,4,10,\dots,2n-6],\dots,[2,\dots,2n-12,2n-6]\}\\
        &= &\text{St}_{D'_n}(2n-6)\cup[2,4,\dots,2n-8]
    \end{eqnarray*}
The intersection $D_n''$ of these two pieces is
    \begin{eqnarray*}
        D''_n &= &\text{St}_{D'_n}(2n-6)\cap[2,4,\dots,2n-8] \\
        & =& \{[6,\dots,2n-8],[2,8,\dots,2n-8],[2,4,10,\dots,2n-8],\\
        & & \quad \dots,[2,\dots,2n-14,2n-8],[2,\dots,2n-12]\}\\
        &= &D_{n-3}
    \end{eqnarray*}
Therefore, by Proposition \ref{wedgeprop},
    \begin{eqnarray*}
    D_n&\simeq &\text{St}_{D_n}(2n-2)\vee[2,\dots,2n-6]\vee\Sigma^2(D_{n-3})\\
    &\simeq & \ast\vee\ast\vee\Sigma^2(D_{n-3})\\
    &\simeq & \Sigma^2(D_{n-3})
    \end{eqnarray*}
\end{proof}

\noindent It remains to compute the three base cases to determine the homotopy type of $E_n$.

\subsection*{The $n\equiv 0\bmod 3$ case} Let $n=6,k=2$. The generating simplices of $B_6$ are $\sigma_6=[2,4,6,8],\tau_6=[4,6,8,10],[2,4,10],$ and $[2,8,10]$. Then we have $B_6= \text{St}_{B_6}(10)\cup\sigma_6,$  with intersection
    \begin{eqnarray*}
        B'_6 &= &\text{St}_{B_6}(10)\cap\sigma_6 \\
        &= & \{[4,6,8],[2,4],[2,8]\}\\
        &=  & \text{St}_{B'_6}(8)\cup [2,4]
    \end{eqnarray*}
Taking a second intersection, 
    $$B''_6=\text{St}_{B'_6}(8)\cap [2,4]=\{[2],[4]\}\simeq\zs^0$$
So we have
    \begin{eqnarray*}
    B_6 &\simeq & \text{St}_{B_6}(10)\vee\sigma_6\vee\Sigma^2(B''_6)\\
    & \simeq & \ast\vee\ast\vee\Sigma^2(\zs^0)\\
    & \simeq &\zs^2 \\
    & = &\zs^{2k-2}
    \end{eqnarray*}
With the base case proven, the inductive step holds by Lemma \ref{lem:bn}:
    \begin{eqnarray*}
        B_n &\simeq &\Sigma^2(B_{n-3})\\
        &\simeq &\Sigma^2(\zs^{2k-4})\\
        &\simeq &\zs^{2k-2}
    \end{eqnarray*}
    
Next we have $$D_6=\{[2,4,6],[6,8,10],[2,8,10],[2,4,10]\}=\text{St}_{D_6}(10)\cup[2,4,6]$$ Taking the intersection gives
$$D'_6=\text{St}_{D_6}(10)\cap[2,4,6]=\{[2,4],[6]\}\simeq\zs^0$$
By Proposition \ref{wedgeprop}, $$D_6\simeq~\text{St}_{D_6}(10)\vee[2,4,6]\vee\Sigma(\zs^0)\simeq\zs^1=\zs^{2k-3}$$

With the base case proven, the inductive step holds by Lemma \ref{lem:dn}:
    \begin{eqnarray*}
        D_n &\simeq &\Sigma^2(D_{n-3})\\
        &\simeq &\Sigma^2(\zs^{2k-5})\\
        &\simeq &\zs^{2k-3}
    \end{eqnarray*}

Putting it all together, for $n=3k$ we have 
\begin{eqnarray*}
    E_n &\simeq & A_n\vee B_n\vee\Sigma(A_n\cap B_n)\\
    &\simeq & B_n\vee\Sigma(D_n)\\
    &\simeq &\zs^{2k-2}\vee\Sigma(\zs^{2k-3})\\
    &\simeq &\zs^{2k-2}\vee\zs^{2k-2}
\end{eqnarray*}

\subsection*{The $n\equiv 1\bmod 3$ case} Let $n=7,k=2$. The generating simplices of $B_7$ are $\sigma_7=[2,4,6,8,10],\tau_7=[4,6,8,10,12],[2,8,10,12],[2,4,10,12],$ and $[2,4,6,12]$. Then we have $B_7=\text{St}_{B_7}(12)\cup\sigma_7,$  with intersection
    \begin{eqnarray*}
        B'_7 &= &\text{St}_{B_7}(12)\cap\sigma_7 \\
        & = &\{[4,6,8,10],[2,8,10],[2,4,10],[2,4,6]\}\\
        &= &\text{St}_{B'_7}(10)\cup [2,4,6]
    \end{eqnarray*}
Taking a second intersection we have 
    $$B''_7=\text{St}_{B'_7}(10)\cap [2,4,6]=\{[4,6],[2,4]\}\simeq*$$
So we have
    \begin{eqnarray*}
    B_7&\simeq &\text{St}_{B_7}(12)\vee\sigma_7\vee\Sigma^2(B''_7)\\
    & \simeq & \ast\vee\ast\vee\Sigma^2(*)\\
    & \simeq & \ast
    \end{eqnarray*}
With the base case proven, the inductive step holds by Lemma \ref{lem:bn}:
    \begin{eqnarray*}
        B_n &\simeq &\Sigma^2(B_{n-3})\\
        &\simeq & \Sigma^2(*)\\
        &\simeq & \ast
    \end{eqnarray*}

Next we have
\begin{eqnarray*}
    D_7&= &\{[2,4,6,8],[6,8,10,12],[2,8,10,12],[2,4,10,12],[2,4,6,12]\}\\
    &= &\text{St}_{D_7}(12)\cup[2,4,6,8].
\end{eqnarray*}
Taking the intersection gives
\begin{eqnarray*}
D'_7 &= &\text{St}_{D_7}(12)\cap[2,4,6,8] \\
  & = &\{[6,8],[2,8],[2,4,6]\} \\
  & = &\text{St}_{D'_7}(8)\cup[2,4,6].
\end{eqnarray*}
A second intersection yields $$D''_7=\text{St}_{D'_7}(8)\cap[2,4,6]=\{[2],[6]\}\simeq\zs^0$$
By Proposition \ref{wedgeprop}, 
\begin{eqnarray*}
D_7 & \simeq &\text{St}_{D_7}(12)\vee[2,4,6,8]\vee\Sigma^2(\zs^{0}) \\
 &\simeq &\zs^{2} \\
 &= &\zs^{2k-2}
\end{eqnarray*}

With the base case proven, the inductive step holds by Lemma \ref{lem:dn}:
    \begin{eqnarray*}
        D_n &\simeq &\Sigma^2(D_{n-3})\\
        &\simeq &\Sigma^2(\zs^{2k-4})\\
        &\simeq &\zs^{2k-2}
    \end{eqnarray*}

Putting it all together, for $n=3k+1$ we have 
\begin{eqnarray*}
    E_n &\simeq & A_n\vee B_n\vee\Sigma(A_n\cap B_n)\\
    &\simeq & B_n\vee\Sigma(D_n)\\
    &\simeq & \ast\vee\Sigma(\zs^{2k-2})\\
    &\simeq & \zs^{2k-1}
\end{eqnarray*}


\subsection*{The $n\equiv 2\bmod 3$ case} Let $n=8,k=2$. The generating simplices of $B_8$ are $\sigma_8=[2,4,6,8,10,12]$, $\tau_8=[4,6,8,10,12,14]$, and $$[2,8,10,12,14],[2,4,10,12,14],[2,4,6,12,14], [2,4,6,8,14].$$ Then we have $B_8=$~st$_{B_8}(14)\cup\sigma_8,$  with intersection
    \begin{eqnarray*}
        B'_8 &= &\text{St}_{B_8}(14)\cap\sigma_8 \\
        & = &\{[4,6,8,10,12],[2,8,10,12],[2,4,10,12],[2,4,6,12],[2,4,6,8]\}\\
        &=  &\text{St}_{B'_8}(12)\cup[2,4,6,8]
    \end{eqnarray*}
Taking a second intersection we have 
\begin{eqnarray*}
    B''_8 &= &\text{St}_{B'_8}(12)\cap[2,4,6,8] \\
    & = &\{[4,6,8],[2,4,6],[2,8]\} \\
    & \simeq &\zs^{1}
\end{eqnarray*}
So we have
    \begin{eqnarray*}
    B_8 &\simeq &\text{St}(14)\vee\sigma_8\vee\Sigma^2(B''_8)\\
    & \simeq & \ast\vee\ast\vee\Sigma^2(\zs^1)\\
    & \simeq & \zs^3 \\
    &= &\zs^{2k-1}
    \end{eqnarray*}
With the base case proven, the inductive step holds by Lemma \ref{lem:bn}:
    \begin{eqnarray*}
        B_n &\simeq &\Sigma^2(B_{n-3})\\
        &\simeq &\Sigma^2(\zs^{2k-3})\\
        &\simeq & \zs^{2k-1}
    \end{eqnarray*}

Next we have
\begin{eqnarray*}
    D_8 & =&\{[2,4,6,8,10],[6,8,10,12,14],[2,8,10,12,14],[2,4,10,12,14],\\
    & & \quad [2,4,6,12,14],[2,4,6,8,14]\}\\
    &= &\text{St}_{D_8}(14)\cup[2,4,6,8,10]
\end{eqnarray*}
Taking the intersection gives
\begin{eqnarray*}
    D'_8 & = &\text{St}_{D_8}(14)\cap[2,4,6,8,10] \\
    &= &\{[6,8,10],[2,8,10],[2,4,10],[2,4,6,8]\}\\
    &= &\text{St}_{D'_8}(10)\cup[2,4,6,8]
\end{eqnarray*}
A second intersection yields $$D''_8=\text{St}_{D'_8}(10)\cap[2,4,6,8]=\{[6,8],[2,8],[2,4]\}\simeq *$$
By Proposition \ref{wedgeprop}, $$D_8\simeq~\text{St}_{D_8}(14)\vee[2,4,6,8,10]\vee\Sigma^2(*)\simeq *$$

With the base case proven, the inductive step holds by Lemma \ref{lem:dn}:
    \begin{eqnarray*}
        D_n &\simeq & \Sigma^2(D_{n-3})\\
        &\simeq &\Sigma^2(*)\\
        &\simeq & \ast
    \end{eqnarray*}

Putting it all together, for $n=3k+2$ we have 
\begin{eqnarray*}
    E_n &\simeq & A_n\vee B_n\vee\Sigma(A_n\cap B_n)\\
    &\simeq & B_n\vee\Sigma(D_n)\\
    &\simeq & \zs^{2k-1}\vee\Sigma(*)\\
    &\simeq & \zs^{2k-1}
\end{eqnarray*}

Note that each of these three cases is quite different: for $n\equiv 0\bmod 3$, we found topology in both $B_n$ and $D_n$; for $n\equiv 1\bmod 3$, we found topology only in $D_n$; and for $n\equiv 2\bmod 3$, all the topology lives in $B_n$. In summary, we have proved the following.

\begin{lemma}\label{lem:en}
    For $k\ge 2$ we have the following homotopy types:
    $$E_n \simeq \begin{cases}
        \zs^{2k-2}\vee \zs^{2k-2} & n = 3k \\
        \zs^{2k-1} & n=3k+1,3k+2.
    \end{cases}$$
\end{lemma}

We have the same result for the space $O_n$. Finally, note that we have
\begin{eqnarray*}
    \smpure(C_n)& = & (\text{St}(0)\cup\text{St}(6))\cup (\text{St}(1)\cup\text{St}(7))\\
    &\simeq & E_n\vee O_n\vee\Sigma((\text{St}(0)\cup\text{St}(6))\cap (\text{St}(1)\cup\text{St}(7)))\\
    &\simeq & E_n\vee O_n\vee\Sigma(\zs^1)\quad (\textrm{Lemma \ref{lem:intersection}})\\
    &\simeq & E_n\vee O_n\vee\zs^2
\end{eqnarray*}
The calculation of Lemma \ref{lem:en} then finishes the argument. This completes the proof of Theorem \ref{thm:cnpure}.
\end{proof}

\begin{remark}
    Contrast this calculation with that of ${\mathfrak M}(C_n)$ \cite{Kozlov1999} and ${\mathfrak M}_{\textrm{pure}}(C_n)$ \cite{ChariJoswig2005}. These satisfy
    $${\mathfrak M}(C_n) \simeq \begin{cases}
        \zs^{2k-1}\vee\zs^{2k-1}\vee\zs^{3k-2}\vee\zs^{3k-2} & n=3k \\
        \zs^{2k}\vee\zs^{3k-1}\vee\zs^{3k-1} & n=3k+1 \\
        \zs^{2k}\vee\zs^{3k}\vee\zs^{3k} & n=3k+2
    \end{cases}$$
    and
    $${\mathfrak M}_\textrm{pure}(C_n) \simeq \zs^2\vee\zs^{n-2}\vee\zs^{n-2}.$$
    Our complexes are a sort of cross between these two types with the $\zs^2$ factor appearing for all $n$, but with the other factors depending on $n\bmod 3$. 
\end{remark}

The complex $\sm(C_n)$ is more complicated. Indeed, there is no discernible pattern to the homotopy types, in contrast with $\smpure(C_n)$. We used Polymake to compute the homology of $\sm(C_n)$ and determined the homology types displayed in Table \ref{smtable}. There are some double suspension relationships (e.g. $n=3,6$; $n=6,9$; $n=7,10$; $n=14,17$; $n=15,18$) but there is no universal pattern. The spaces for $n\equiv 0\bmod 4$ are clearly different, as one might expect given that the maximal  possible number of stuck pairs in a matching increases at these values. Polymake was unable to complete the calculation beginning at $n=22$.

\begin{table}
\begingroup 
\renewcommand{\arraystretch}{1.4}
\begin{tabular}{r|l|r|l}
\hline
$n$ & $\sm(C_n)$ & $n$ & $\sm(C_n)$ \\
\hline
1 & $\emptyset$  & 12 & $\bigvee_3\zs^5\vee\bigvee_4\zs^6$ \\
2 & $\zs^0$  & 13 & $\zs^6\vee\zs^7\vee\zs^7$ \\
3 & $\bigvee_5 \zs^0$ & 14 & $\zs^6\vee\zs^7\vee\zs^7$ \\
4 & $\zs^1\vee\zs^1\vee\zs^1$ & 15 & $\zs^6\vee\bigvee_4\zs^8$ \\
5 &  $\zs^1    $ & 16 & $\bigvee_3\zs^7\vee\bigvee_2\zs^9$ \\
6 & $\bigvee_5\zs^2$    & 17 & $\zs^8\vee\zs^9\vee\zs^9$  \\   
7 & $\zs^2\vee\zs^3\vee\zs^3$ & 18 & $\zs^8\vee\bigvee_4\zs^{10}$ \\ 
8 & $\bigvee_5\zs^3$  & 19 & $\zs^8\vee\zs^{11}\vee\zs^{11}$ \\
9 & $\bigvee_5\zs^4$ & 20 & $\bigvee_3\zs^9\vee\bigvee_2\zs^{11}$ \\
10 & $\zs^4\vee\zs^5\vee\zs^5$ & 21 & $\zs^{10}\vee\bigvee_4\zs^{12}$ \\
11 & $\zs^4\vee\zs^5\vee\zs^5$ & 22 & out of memory \\
\hline
\end{tabular}
\endgroup
\caption{\label{smtable} The first few homology types of $\sm(C_n)$.}
\end{table}

\section{The $n$-simplex and its boundary}
In this section we study $\sm(\Delta^n)$ and $\sm(\partial\Delta^n)$. These turn out to be related to complexes of directed forests, which were studied by Kozlov \cite{Kozlov1999}. 

Let $G$ be a directed graph with vertex set $V(G)$. We say that $G$ is a {\em directed tree} with root $x\in V(G)$ if for every $y\in V(G)$ there is a unique directed path from $x$ to $y$. A directed graph $G$ is called a {\em directed forest} if there is a decomposition $V(G)=\coprod_{i\in I} A_i$ such that the subgraph induced by each $A_i$ is a directed tree and there are no edges between $A_i$ and $A_j$ for $i\ne j$. 

\begin{definition}
    Let $G$ be a directed graph. Construct a simplicial complex $\Delta(G)$ as follows. The vertices of $\Delta(G)$ are the directed edges of $G$ and faces are all directed forests that are subgraphs of $G$.
\end{definition}

Now consider the complete directed graph $G_n$ on $n$ vertices. This has an edge in each direction between any pair of vertices so that there are $n(n-1)$ edges in $G_n$. 

\begin{theorem}[\cite{Kozlov1999}, Theorem 3.1]\label{kozlovwedge} The complex $\Delta(G_n)$ has the homotopy type of a wedge of $(n-1)^{n-1}$ copies of $\zs^{n-2}$.
\end{theorem}

Now consider the strong Morse complex $\sm(\Delta^n)$. Denote the vertices of $\Delta^n$ by $v_0,v_1,\dots,v_n$. Fix a labeling $\sigma$ of the vertices ($v_i\leftrightarrow \sigma(i)$) and note that for each $j$ the subcomplex $\Delta^n_j$ is a $j$-simplex and is thus a cone on any of its $(j-1)$-faces. In particular, this implies that any strong gradient is a maximal matching on $\Delta^n$, leaving only the vertex labeled with $0$ as critical. It follows that $\sm(\Delta^n)$ is pure of dimension 
$$\frac{1}{2}(2^{n+1}-2)-1=2^n-2.$$

\begin{theorem}\label{strongsimplex}
    The complex $\sm(\Delta^n)$ collapses to a subcomplex $X_n$ of dimension $n-1$ isomorphic to $\Delta(G_{n+1})$. Consequently, $\sm(\Delta^n)$ has the homotopy type of a wedge of $n^n$ copies of $\zs^{n-1}$.
\end{theorem}

\begin{proof}
    The crucial observation is the following: any strong matching on $\Delta^n$ is completely determined by the $n$ vertex-edge pairs. Indeed, the subcomplex $\Delta^n_j$ is the $j$-simplex spanned by the vertices labeled $0,1,\dots,j-1,j$. The lower link of vertex $j$ is then the simplex spanned by the vertices labeled $0,1,\dots,j-1$ and hence the vertex $v_{\sigma(j)}$ is downward dominated by any of the vertices with smaller labels. This allows for $j$ edge choices for pairing with vertex $v_{\sigma(j)}$, but that choice uniquely determines the rest of the pairs in $\Delta^n_j$. Thus, all pairings of $i$- and $(i+1)$-simplices for $i\ge 2$ are completely determined by the choice of vertex-edge pairings.

    Let $X_n$ be the subcomplex of $\sm(\Delta^n)$ generated by the simplices whose vertices correspond to the vertex-edge pairs. Since there are $n$ such pairs in any gradient, $\dim X_n=n-1$. By the observation above, each such simplex $\sigma$ determines a {\em unique} facet $\tau_\sigma$ of $\sm(\Delta^n)$, and all facets arise in this manner. 

     We again make use of generalized discrete Morse theory.  Define a partition of $\sm(\Delta^n)\setminus X_n$ as follows.  For each simplex $\tau\in\sm(\Delta^n)\setminus X_n$ having at least one vertex in $X_n$, let $\alpha_\tau\subset \tau$ be the face generated by the vertices not in $X_n$. An argument similar to that in the proof of Lemma \ref{lem:endefretract} shows that the intervals $[\alpha_\tau,\tau]$ partition $\sm(\Delta^n)\setminus X_n$ and so, by letting each simplex of $X_n$ be critical, we have a generalized discrete gradient on $\sm(\Delta^n)$. By Theorem \ref{gencollapse} we see that $\sm(\Delta^n)\searrow X_n$.

    Now observe that a maximal vertex-edge matching is precisely a directed spanning tree in the $1$-skeleton of $\Delta^n$. The collection of all of these corresponds to the collection of all directed trees in the complete directed graph $G_{n+1}$. Thus $X_n$ is isomorphic to $\Delta(G_{n+1})$ and hence by Theorem \ref{kozlovwedge} we see that $$\sm(\Delta^n)\simeq \bigvee_{n^n} \zs^{n-1}.$$
\end{proof}

\begin{remark}
    We may define a generalized discrete Morse function for the generalized gradient in the proof of Theorem \ref{strongsimplex} as follows. The $n(n+1)$ vertices of $X_n$ correspond to the pairs $ij$ for $0\le i\ne j\le n$. Put these in lexicographical order and then number them beginning at $0$. The remaining vertices of $\sm(\Delta^n)$ correspond to pairs $(\sigma^{(p)},\tau^{(p+1)})$ with $p\ge 1$. We order these as follows: each simplex $\sigma = [v_{i_0},v_{i_1},\dots,v_{i_p}]$ is determined by the string $i_0,i_1,\dots,i_p$, listed in increasing order. If $\sigma$ is paired with $\tau = [v_{j_0},v_{j_1},\dots,v_{j_{p+1}}]$, where $\{i_0,\dots,i_p\}\subset\{j_0,\dots,j_{p+1}\}$, then the vertex of $\sm(\Delta^n)$ corresponding to the pair $(\sigma,\tau)$ is labeled $i_0i_1\dots i_pj_0j_1\dots j_{p+1}$. Put these strings in lexicographical order and then number them beginning with $n(n+1)$, which corresponds to the pair $[v_0,v_1]\subset [v_0,v_1,v_2]$.  
    
    We now define $f:\sm(\Delta^n)\setminus X_n\to\zr$. For the vertices $v$ in $X_n$ set $f(v) = 0$. For the vertices in $\sm(\Delta^n)\setminus X_n$ set $f(v)=v$. Extend $f$ to all simplices in $\sm(\Delta^n)\setminus X_n$ by setting $f([v_{i_0},v_{i_1},\dots,v_{i_p}]) = \sum_j f(v_{i_j})$. Then $f$ is constant on each interval $[\alpha_\tau,\tau]$ and $f$ satisfies $f(\mu)\le f(\nu)$ whenever $\mu$ is a face of $\nu$, with equality only if $\mu$ and $\nu$ lie in a common interval. If we set $f(\sigma) = \dim\sigma - N$ for some integer $N>\dim X_n = n-1$, then each simplex in $X_n$ is critical for $f$. The corresponding discrete gradient is therefore the one described above.
\end{remark}

We now turn our attention to $\sm(\partial\Delta^n)$. The construction of a strong gradient proceeds just as in the case of $\Delta^n$ until the final stage. When considering the vertex labeled with $n$, the subcomplex $\partial\Delta^n_n$ is the entire space $\partial\Delta^n$ and the lower link of this vertex is the boundary of the opposite $(n-1)$-simplex, which is {\em not} a cone on any of its vertices.

\begin{theorem}\label{strongboundary}
    The complex $\sm(\partial\Delta^n)$ collapses to a subcomplex $Y_n$ of dimension $n-2$. The complex $Y_n$ has the homotopy type of a wedge of $f(n)$ copies of $\zs^{n-2}$, where
    \begin{equation*}
    f(n) = \frac{n^2 + 3n}{2} + \sum_{k=1}^{n-2} \binom{n+1}{k}(n-k)^{n-k}
    \end{equation*}
\end{theorem}

\begin{proof}
    By the observation above, any strong matching on $\partial\Delta^n$ consists of a maximal matching on the $(n-1)$-simplex opposite the vertex labeled $n$. In turn, such a matching is completely determined by the vertex-edge pairs in that $(n-1)$-simplex. Let $Y_n$ be the subcomplex of $\sm(\partial\Delta^n)$ consisting of simplices whose vertices correspond to the vertex-edge pairs in the $(n+1)$ $(n-1)$-faces in $\partial\Delta^n$. As in Theorem \ref{strongsimplex}, we see that $\sm(\partial\Delta^n)\searrow Y_n$ and $\dim Y_n=n-2$.

    We claim that $Y_n$ has the homotopy type of a wedge of $(n-2)$-spheres. For $i=0,1,\dots,n$, let $A_i\subset Y_n$ be the subcomplex corresponding to gradients that do not involve vertex $v_i$. Note that each $A_i$ is isomorphic to the complex $X_{n-1}$ defined in the proof of Theorem \ref{strongsimplex}, and is therefore a wedge of $(n-2)$-spheres. Given a subset $J\subset\{0,1,\dots,n\}$, observe that
    $$A_J = \bigcap_{j\in J}A_j \simeq \sm(\Delta^{n-|J|}),$$
    so that each $A_J\simeq \bigvee \zs^{n-|J|-1}$. For $|J|=n,n+1$, we get $A_J=\emptyset$ since $\sm(\Delta^0)=\emptyset$. 
    
    Note that $Y_n=A_0\cup A_1\cup\dots A_n$. For any subsets $J'\subset J\subset \{0,1,\dots,n\}$ we see that the inclusion $A_{J}\hookrightarrow A_{J'}$ is a map $\bigvee \zs^{n-|J|-1} \to \bigvee \zs^{n-|J'|-1}$ and is therefore null-homotopic. Iterating Proposition \ref{wedgeprop}, we conclude that
    \begin{eqnarray*}
        Y_n & = & A_0\cup A_1\cup\dots\cup A_n \\
         & \simeq & \bigvee_{i=0}^n A_i \vee \bigvee_{i,j}\Sigma(A_i\cap A_j) \vee \bigvee_{i,j,k}\Sigma^2(A_i\cap A_j\cap A_k) \vee \cdots \\
         &  &  \vee \bigvee_{|J|=k+1}\Sigma^k(A_J)\vee \cdots \vee \bigvee_{|J|=n-1} \Sigma^{n-2}(\bigcup A_J)
    \end{eqnarray*}
    The last term consists of wedges of $(n-2)$-fold suspensions of unions of various $A_J$ for $|J|=n-1$; these are finite collections of points and so the iterated suspension is a wedge of spheres of dimension $n-2$. All the other terms are also wedges of $(n-2)$-spheres and so we conclude that $Y_n$ is as well.

    It remains to show that there are $f(n)$ spheres in the wedge. For this, we turn to the Mayer-Vietoris spectral sequence for computing the cohomology of $Y_n$. This sequence has
    $$E_1^{pq} = \bigoplus_{|J|=p+1} H^q(A_J)\Longrightarrow H^{p+q}(Y_n).$$ Since each $A_J$ has the homotopy type of a wedge of spheres, this sequence is concentrated in the row $q=0$ and the diagonal $p+q=n-2$. As a result $E_2=E_\infty$ and the only nontrivial differentials are $d_1:E_1^{p,0}\to E_1^{p+1,0}$. The $q=0$ row has the following form
    $$0\to \zz^{\binom{n+1}{1}} \to \zz^{\binom{n+1}{2}}\to \cdots \to \zz^{\binom{n+1}{n-2}}\to (\zz^2)^{\binom{n+1}{n-1}}\to 0$$
    The $\zz^2$ in the last term arises from the fact that $\sm(\Delta^1)\simeq \zs^0$ and so has two components (this term is a sum of copies of $H^0(\zs^0)$). We know that the $E_2$-terms must satisfy
    $$E_2^{00}\cong\zz;\quad E_2^{p0}=0, 0<p<n-2;\quad E_2^{n-2,0}=\zz^x$$ and it remains to determine $x$. 

    Recall that the alternating sum of the binomial coefficients for a fixed $n$ is zero:
    $$0=\binom{n+1}{0} - \binom{n+1}{1} +\binom{n+1}{2}-\cdots + (-1)^n\binom{n+1}{n} + (-1)^{n+1}\binom{n+1}{n+1}$$
    Rearranging terms we find that
    $$\binom{n+1}{1}-\binom{n+1}{2}+\cdots+(-1)^{n-1}\binom{n+1}{n-2} +(-1)^n\binom{n+1}{n-1} $$
    $${}=\binom{n+1}{0} + (-1)^n\binom{n+1}{n} + (-1)^{n+1}\binom{n+1}{n+1}$$
    $${}=1+(-1)^n(n+1) +(-1)^{n+1}$$
    Now, the Euler characteristic $\chi$ of the complex $E_1^{p0}$ is given by the alternating sum of the ranks of the various groups. This is almost the top line above; all that is missing is an additional $(-1)^n\binom{n+1}{n-1}$. Adding this to both ends of this chain of equalities yields
    $$\chi = 1+(-1)^n(n+1) +(-1)^n\binom{n+1}{n-1} + (-1)^{n+1}.$$
    On the other hand, we know that $\chi = 1+(-1)^nx$. Thus we find
        $$(-1)^nx  =  (-1)^n(n+1) + (-1)^n\binom{n+1}{n-1} + (-1)^{n+1}$$ and so 
    \begin{eqnarray*}
        x & = & (n+1) + \binom{n+1}{n-1} - 1 \\
           & = & n + \frac{n(n+1)}{2} \\
           & = & \frac{n^2 + 3n}{2}
    \end{eqnarray*}

    Finally, note that for $0\le p <n-2$, the group $E_1^{pq}$, $p+q=n-2$ has rank $$\binom{n+1}{p+1}(n-p-1)^{n-p-1},$$ yielding the formula for $f(n)$.
    
\end{proof} 

\begin{remark}
    A table of the first few values of $f(n)$ appears in Table \ref{fntable}. We used Polymake to compute the homology of $Y_n$ directly for $n=3,4,5,6,7$; the results agreed with the values in the table. The calculation for $n=7$ took approximately 13 hours on a MacBook Pro with an Apple M2 Pro chip and 16GB of RAM.
\end{remark}

\begin{table}
\begingroup
\renewcommand{\arraystretch}{1.4}
\begin{tabular}{r|r|r|r}
\hline
$n$ & $f(n)$ & $n$ & $f(n)$ \\
\hline
3 & 25               & 8  & 9{,}390{,}041         \\
4 & 189              & 9  & 211{,}157{,}281        \\
5 & 2{,}041          & 10 & 5{,}337{,}224{,}991    \\
6 & 28{,}363         & 11 & 149{,}708{,}449{,}705  \\
7 & 477{,}233        & 12 & 4{,}612{,}995{,}361{,}285 \\
\hline
\end{tabular}
\endgroup
\caption{\label{fntable} The first few values of $f(n)$.}
\end{table}

\section{Miscellaneous results}\label{sec:misc} Complexes of discrete Morse functions are not well-understood in general. Once a simplicial complex $K$ has more than a few simplices, the dimension of ${\mathfrak M}(K)$ becomes rather large. For example, $\dim {\mathfrak M}(\Delta^n) = 2^n-2$, and its homotopy type is known only for $n\le 3$ \cite{Scoville2026}. Our complexes are smaller, but as the preceding calculations show, their homotopy types are still a challenge to compute.

Some connectivity results are known. For example, there is the following theorem.

\begin{theorem}[\cite{ScovilleZaremsky2022}, Theorem 2.7]
    If $K$ has a vertex of degree $d$ in $K^{(1)}$, then ${\mathfrak M}(K)$ is $(d-2)$-connected.
\end{theorem}

This result can be made sharper for graphs (\cite{ScovilleZaremsky2022}, Theorem 4.3). As the next result shows, however, there is no hope for a similar result for the complexes $\sm(K)$.

Denote the complete undirected graph on $n$ vertices by $K_n$. This has $n(n-1)/2$ undirected edges. Observe that ${\mathfrak M}(K_n)$ is the complex $\Delta(G_n)$ (Theorem \ref{kozlovwedge}), which has the homotopy type of a wedge of $(n-2)$-spheres. The complex $\sm(K_n)$ has a much different structure, however.

\begin{proposition}\label{completestrong}
    The complex $\sm(K_n)$ is equal to the $0$-skeleton of $\Delta(G_n)$.
\end{proposition}

\begin{proof}
    We may assign the value $0$ to any vertex $v_j\in K_n$. Labeling any other vertex $v_i$ with $1$ will create a downward dominated vertex, and so we get the vertex $ij\in \sm(K_n)$ for all $n(n-1)$ pairs $1\le i\ne j\le n$. Thus, the vertices of $\sm(K_n)$ form the entire $0$-skeleton of $\Delta(G_n)$. But now, when we label a third vertex $v_k$ with $2$, the lower link will consist of $v_i$ and $v_j$ (but not the edge joining them). This is not a cone. Inductively, when we assign the label $\ell>2$ to any vertex, the lower link will consist of the previous $\ell$ vertices, and is therefore not a cone. Thus, the only strong Morse matchings on $K_n$ are the single vertex-edge pairs and hence $\dim\sm(K_n)=0$. 
\end{proof}

So, even though the degree of a vertex in $K_n$ grows linearly with $n$ (and hence so does the connectivity of ${\mathfrak M}(K_n)$), $\sm(K_n)$ is not even connected for any $n\ge 2$.

We also have the following result. Compare with \cite{DonovonLinScoville2023}, Proposition 4.2.

\begin{proposition}\label{disjointunion}
    Let $K$ and $L$ be simplicial complexes with at least one edge. Then $\sm(K\sqcup L) = \sm(K)\ast\sm(L)$.
\end{proposition}

\begin{proof}
    The condition that $K$ and $L$ have at least one edge simply guarantees that the complexes $\sm(K)$ and $\sm(L)$ are nonempty. Given an ordering $\sigma$ of the vertices of $K$ and an ordering $\tau$ of the vertices of $L$, denote the corresponding strong discrete gradients by $V_\sigma\in\sm(K)$ and $V_\tau\in\sm(L)$. Then $V_\sigma\cup V_\tau$ is a strong gradient on $K\sqcup L$ induced by the ordering $(\sigma,\tau)$ on the vertices of $K\sqcup L$ and hence $\sm(K)\ast\sm(L)\subseteq\sm(K\sqcup L)$.

    Conversely, let $\alpha$ be an ordering of the vertices of $K\sqcup L$. Since $K$ and $L$ are disjoint, the construction of a strong discrete gradient on $K\sqcup L$ requires consideration of lower links that lie entirely in either $K$ or $L$. That is, $\alpha$ induces orders $\sigma$ and $\tau$ on $V(K)$ and $V(L)$, respectively and the strong gradient $V_\alpha$ is the union of $V_\sigma\in\sm(K)$ and $V_\tau\in\sm(L)$. Thus, $\sm(K\sqcup L)\subseteq \sm(K)\ast\sm(L)$.
\end{proof}

In particular, note that $\sm(K\sqcup P_1)=\Sigma(\sm(K))$.

Another result from \cite{DonovonLinScoville2023} is the following. Recall that a {\em leaf} in a simplicial complex is an edge $\{u,uv\}$ with $\deg(u)=1$. That is, the edge $uv$ is attached to $K\setminus \{u\}$ at the vertex $v$. 

\begin{lemma}[\cite{DonovonLinScoville2023},Lemma 3.8]\label{wedgelemma}
    Let $K$ be a simplicial complex with leaf $\{a,ab\}$ and $c$ a neighbor of $b$ not equal to $a$. Then the vertex $\{b,bc\}\in{\mathfrak M}(K)$ is dominated by the vertex $\{a,ab\}$.
\end{lemma}

The authors then use this to prove that $${\mathfrak M}(C_n\vee\ell) \sco {\mathfrak M}(P_{n-1}\sqcup\ell)$$ for a leaf $\ell$. The latter space has the homotopy type of the suspension of ${\mathfrak M}(P_{n-1})$. The corresponding result is {\em not} true for $\sm(K)$, however, as the following example shows.

\begin{example}
    Consider the space $C_3\vee \ell$ shown on the left in Figure \ref{fig:cyclewedge}. The complex $\sm(C_3\vee\ell)$ is shown on the right. Note that the vertices of ${\mathfrak M}(C_3\vee\ell)$ and $\sm(C_3\vee\ell)$ are the same; they correspond to the $8$ possible vertex-edge pairings and are numbered $0-7$ corresponding to the lexicographical ordering on the pairs. According to Lemma \ref{wedgelemma}, vertex $7$, which corresponds to the leaf pairing $\{v_3,v_3v_2\}$, dominates vertex $1$, which corresponds to the pair $\{v_0,v_0v_2\}$, in ${\mathfrak M}(C_3\vee\ell)$. However, this is obviously not the case in $\sm(C_3\vee\ell)$. Moreover, while ${\mathfrak M}(C_3\vee\ell)\simeq \ast$, we clearly have $\sm(C_3\vee\ell)\simeq \zs^1$, and this is not the same as $\sm(P_2\sqcup\ell)=\Sigma\sm(P_2)\simeq\ast$.
\end{example}

\begin{figure}
   
\centering
\begin{tikzpicture}[scale=1]
  \tikzset{
    vertex/.style={circle, fill=black, draw=black, inner sep=0pt, minimum size=5pt},
    edge/.style={thick, black}
  }

  
    \node[vertex, label=left:\(v_0\)] (v0_g2) at (0, 1.732) {};
    \node[vertex, label=right:\(v_1\)] (v1_g2) at (2.0, 1.732) {};
    \node[vertex, label=right:\(v_2\)] (v2_g2) at (1.0, 0) {};
    
    \node[vertex, label=right:\(v_3\)] (v3_g2) at (1.0, -1.5) {};

    \draw[edge] (v0_g2) -- (v1_g2);
    \draw[edge] (v1_g2) -- (v2_g2);
    \draw[edge] (v2_g2) -- (v0_g2);
    \draw[edge] (v2_g2) -- (v3_g2);

\begin{scope}[xshift=6.5cm, yshift=0.0cm]
  \node[vertex, label={[above left, xshift=5pt, yshift=1pt]:7}] (v7) at (0,0) {};
  
  \node[vertex, label=above:3] (v3) at (120:1.8) {};
  \node[vertex, label=above right:4] (v4) at (60:1.8) {};
  \node[vertex, label=right:5] (v5) at (0:1.8) {};
  \node[vertex, label=below right:2] (v2) at (300:1.8) {};
  \node[vertex, label=below:0] (v0) at (270:1.8) {};
  \node[vertex, label=below:1] (v1) at (240:1.8) {};
  
  \node[vertex, label=left:6] (v6) at (-1.8, 0) {};

  \draw[edge] (v7) -- (v3);
  \draw[edge] (v7) -- (v1);
  \draw[edge] (v3) -- (v6);
  \draw[edge] (v1) -- (v6);
  \draw[edge] (v7) -- (v4);
  \draw[edge] (v7) -- (v5);
  \draw[edge] (v7) -- (v2);
  \draw[edge] (v7) -- (v0);
 \end{scope}
 
\end{tikzpicture}

    \caption{\label{fig:cyclewedge} The space $C_3\vee\ell$ (left), and the complex $\sm(C_3\vee\ell)$ (right).}
\end{figure}
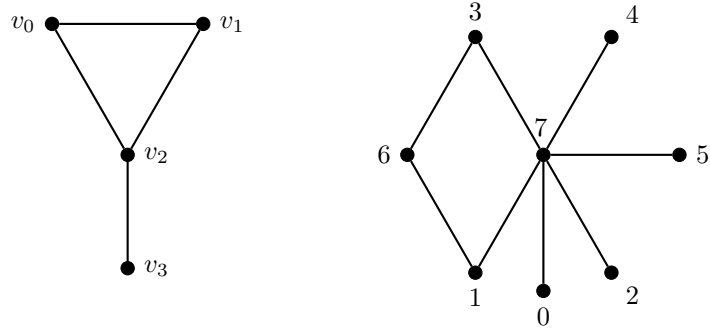

\subsection*{Some related complexes} Recall that the complex $\sm(P_n)$ has $2n$ vertices $0,1,2,\dots ,2n-1$, and the simplices are characterized by having no pair $\{i,i+1\}$ for all $i$ and $\{j,j+3\}$ for all odd $j$. Let $X_m$ be the complex with vertices $0,1,2,\dots ,m-1$ subject to these same restrictions: a simplex cannot contain a pair $\{i,i+1\}$ or a pair $\{j,j+3\}$ for $j$ odd. Note that $X_{2m} = \sm(P_m)$. We have the following result.

\begin{theorem}\label{xmthm}
    The complexes $X_m$ have the following homotopy types ($k\ge 0)$:
    $$X_m\simeq \begin{cases}
        \ast & m = 8k+1, 8k+4, 8k+5, 8k+6 \\
        \zs^{2k} & m=8k+2, 8k+3 \\
        \zs^{2k+1} & m=8k+7,8k+8.
    \end{cases}$$
\end{theorem}

\begin{proof}
    We compute the $k=0$ case directly. The space $X_1$ is a point; $X_2$ is two points; $X_3$ has three vertices $\{0,1,2\}$ with the single edge $[0,2]$ and is therefore homotopic to $\zs^0$. The space $X_4$ is the space $\sm(P_2)$, which is contractible. The space $X_5$ has vertices $\{0,1,2,3,4\}$, edges $[0,3], [1,3]$, and $2$-simplex $[0,2,4]$; this is contractible. The space $X_6$ is $\sm(P_3)$, which is contractible. The space $X_7$ has facets $[0,2,4,6]$, $[0,2,5]$, $[0,3,5]$, and $[1,6]$; this is homotopic to $\zs^1$ (see Figure \ref{p4diagram}). The space $X_8$ is $\sm(P_4)$, which is homotopic to $\zs^1$.

    Now the proof proceeds exactly like the proof of Theorem \ref{pathhtyptype}: Assume $m\ge 9$ and write $X_m=\text{St}(0) \cup \text{St}(13) \cup \text{St}(16)$. Then we conclude that $X_m\simeq \Sigma^2(\text{St}(0) \cap \text{St}(13) \cap \text{St}(16))$, and this space is $\Sigma^2X_{m-8}$. This completes the proof.
\end{proof}

\begin{remark}
    Given the $4$-periodic behavior exhibited in Theorem \ref{pathhtyptype}, together with the fact that $\sm(P_m)=X_{2m}$, this period $8$ phenomenon is not surprising.
\end{remark}

We could also consider the complex $Y_m$ defined in the same way, but without the restriction that $j$ be odd to exclude pairs $\{j,j+3\}$; that is, we disallow all such pairs as vertices of simplices. Note that $Y_m\subset X_m$. This is a more symmetric object, but there is no obvious connection to the other complexes we have considered in this paper. 

\begin{theorem}\label{ymthm}
    The complexes $Y_m$ have the following homotopy types ($k\ge 0)$:
    $$Y_m\simeq \begin{cases}
        \ast & m=5k+1 \\
        \zs^k & m=5k+2, 5k+3, 5k+4, 5k+5.
    \end{cases}$$
\end{theorem}

\begin{proof}
    The proof is similar to that of Theorem \ref{xmthm}, but simpler. We first compute the $k=0$ case directly. The space $Y_1$ is a point; $Y_2$ is two points; $Y_3$ is the interval $[0,2]$ and the vertex $[1]$, so it is homotopy equivalent to $\zs^0$. The space $Y_4$ consists of two disjoint intervals $[0,2]$ and $[1,3]$ and is therefore homotopy equivalent to $\zs^0$. The space $Y_5$ is the disjoint union of $[0,2,4]$ and $[1,3]$; this is also $\zs^0$ up to homotopy.

    Now assume $m\ge 6$ and note that $Y_m=\text{St}(0)\cup \text{St}(1)$. It follows that $Y_m\simeq \Sigma(\text{St}(0)\cap\text{St}(1))$. Observe that any simplex in $\text{St}(0)\cap\text{St}(1)$ cannot involve vertices $0,1,2,3,4$: having $0$ excludes $1$ and $3$, and having $1$ excludes $0$, $2$, and $4$. Thus every facet of the intersection has initial vertex $5$ or $6$. If a facet begins with $5$ then it can be extended by $[0]$ and $[1,3]$ and if it begins with $6$ it can be extended by $[0,2,4]$ and $[1]$. It follows that $\text{St}(0)\cap\text{St}(1)$ is isomorphic to $Y_{m-5}$; that is, $Y_m\simeq \Sigma Y_{m-5}$. This completes the proof.
\end{proof}

\begin{remark}
    Note that the complexes $X_m$ and $Y_m$ are independence complexes of certain graphs. Recall that a set of vertices in a graph $G$ is {\em independent} if no two are adjacent in $G$. The {\em independence complex} $I_G$ is the simplicial complex with vertex set $V(G)$ whose simplices consist of the independent sets of vertices. These objects have been studied extensively in the combinatorics literature (e.g. \cite{Barmak2013},\cite{Jonsson2011},\cite{Kozlov1999}). Define a graph $G(X_m)$ with vertex set $0,1,\dots ,m-1$ and edges $\{i,i+1\}$ for all $i=0,\dots ,m-2$ and $\{j,j+3\}$ for all odd $1\le j\le m-1$. The graph $G(Y_m)$ is defined by adding edges $\{j,j+3\}$ to $G(X_m)$ for all even $0\le j\le m-1$. Let $V$ be the set of vertices with even labels and $W$ the set of vertices with odd labels. Then both $G(X_m)$ and $G(Y_m)$ are bipartite with parts $V$ and $W$ and the corresponding independence complexes are $X_m$ and $Y_m$, respectively. According to Theorem 3.1 of \cite{Jonsson2011}, these spaces have the homotopy type of a suspension of a particular simplicial complex $\Gamma_{G,V}$ with vertex set $V$. The interested reader is invited to verify Theorems \ref{xmthm} and \ref{ymthm} in this context.
\end{remark}

\bibliographystyle{plain}
\bibliography{SDMT.bib}

\end{document}